\documentclass[a4paper]{article}
\usepackage[all]{xy}\usepackage[latin1]{inputenc}        
\usepackage[dvips]{graphics,graphicx}
\usepackage{amsfonts,amssymb,amsmath,color,mathrsfs, amstext}
\usepackage{amsbsy, amsopn, amscd, amsxtra, amsthm,authblk}
\usepackage{enumerate,algorithmicx,algorithm}
\usepackage{algpseudocode}
\usepackage{upref}
\usepackage{geometry}
\usepackage{subfigure}
\usepackage[displaymath]{lineno}
\usepackage{float}
\usepackage{yhmath}
\usepackage{booktabs}
\usepackage{multirow}
\usepackage{makecell}
\usepackage[colorlinks,
            linkcolor=blue,
            anchorcolor=blue,
            citecolor=blue
            ]{hyperref}

\numberwithin{equation}{section}

\def\e{\epsilon}

\newcommand{\tcb}[1]{\textcolor{blue}{#1}}

\newtheorem{theorem}{Theorem}[section]
\newtheorem*{theorem*}{Theorem}
\newtheorem{lemma}{Lemma}[section]
\newtheorem{definition}{Definition}[section]
\newtheorem{remark}{Remark}[section]
\newtheorem{proposition}{Proposition}[section]

\numberwithin{equation}{section}

\begin{document}

\title{Numerical stability of Gr\"unwald-Letnikov method for time fractional delay differential equations}

\author[1]{Lei Li\thanks{E-mail: leili2010@sjtu.edu.cn}}
\author[2]{Dongling Wang\thanks{E-mail: wdymath@nwu.edu.cn}}
\affil[1]{School of Mathematical Sciences, Institute of Natural Sciences, MOE-LSC, Shanghai Jiao Tong University, Shanghai, 200240, P. R. China.}
\affil[2]{Department of Mathematics and Center for Nonlinear Studies, Northwest University, Xi'an, Shaanxi, 710127, P. R. China.}
\maketitle

\begin{abstract}
This paper is concerned with the numerical stability of time fractional delay differential equations (F-DDEs) based on Gr\"{u}nwald-Letnikov (GL) approximation (also called fraction backward Euler scheme) for the Caputo fractional derivative, in particular, the numerical stability region and the Mittag-Leffler stability. Using the boundary locus technique, we first derive the exact expression of the numerically stability region in the parameter plane, and show that the fractional backward Euler scheme based on GL scheme is not $\tau(0)$-stable, which is different from the backward Euler scheme for integer DDE models. Secondly, we also prove the numerical Mittag-Leffler stability for the numerical solutions provided that the parameters fall into the numerical stability region, by employing the singularity analysis of generating function. Our results show that the numerical solutions of F-DDEs are completely different from the classical integer order DDEs, both in terms of $\tau(0)$-stabililty and the long-time decay rate. 

\end{abstract}


\section{Introduction}\label{Introd}

%
%
%
%
%
%

Time fractional differential equations, or time nonlocal differential equations, have received great attention and research in recent years. The most important reason is that such equations are often more accurate than classical integer-order equations in describing various physical processes with inherited or memory characteristics. Time fractional differential operators are some kind of convolution integral operators \cite{brunner2017volterra, diethelm10, li2018generalized,liliu2018compact}, so they have non-local nature.  They usually include the classical integral derivative as a special case. The time fractional-order equations can usually be equivalent to Volterra integral equations under some conditions \cite{brunner2017volterra, diethelm10}.

Note that the physical interpretation of fractional order calculus is not as intuitive as integer order calculus, but some authors have tried to give reasonable physical explanations. Here we would like to mention the explanation given by continuous time random walk model based on probability theory (see, for example, \cite{chen2017time}), which provides a very natural mathematical basis and physical explanation for fractional derivatives operator from anomalous diffusion \cite{metzler2000random}. Also, the generalized Langevin equations together with fractional noise yield Caputo derivatives naturally \cite{burov2008fractional,kou2008stochastic,li2017fractional}. On the other hand, the delay effect is widespread in various practical models \cite{bellen2013numerical, brunner2017volterra}. Time delays occur in various models due to the time needed for material, energy, and information to be transported between different parts of a system.  For example, delays are often used to describe incubation time in biological models. Therefore, if the memory characteristics and delay effects of the model are taken into account at the same time, various time fractional order delay differential equations could be obtained, such as time fractional SIRI epidemic model with relapse and a general nonlinear incidence rate \cite{lahrouz386mittag}. 

For time fractional differential equations, just like the standard integer order equations, once we know the existence and uniqueness of the solutions, then the most important thing is to study various qualitative properties of the solutions. There are two typical differences between fractional differential equations and standard integer order equations.  One is when time tends to be the initial time, the other is when time tends to infinite. 

When time tends to be the initial time, the solutions of time fractional differential equations usually have low regularity in time, and are generally 
H\"{o}lder continuous with order $\alpha$ but without the first order derivative \cite{brunner2017volterra, diethelm10}. 
By using operator-valued Fourier multiplier results on vector-valued H\"{o}lder continuous function spaces,  a necessary and sufficient condition of the $C^{\alpha}$-well-posedness for F-DDEs is proved in \cite{bu2018well}.
The low regularity of the solutions near initial time poses a serious numerical challenge to derive high order schemes for time fractional equations. 
Some of the most important advances for time fractional differential equations can be seen in \cite{dabiri2018numerical, kaslik2012analytical,  maleki2019fractional,zayernouri2014spectral,zhao180generalized,zhang2019asymptotic}.
Another problem related to this is to assign reasonable initial conditions to time fractional order equations. 
As pointed out in \cite{garrappa2020initial}, this is generally an open and debated issue and earning considerable attention. 
Nevertheless, this does not preclude various applications of time fractional order equations.
This is not the focus of this article. We will take the Caputo fractional derivative and impose the standard initial function for F-DDEs.

When time tends to infinity, the solutions of the time fractional order equations exhibit a completely different asymptotic behavior from that of the integer order equations.  The standard integer order equations usually converge exponentially to the equilibrium state, while the time fractional order equations often have only algebraic decay rate, which is $O(t^{-\alpha})$ as $t\to+\infty$ and leads to the so-called Mittag-Leffler stability. See the exact definition below in section \ref{sec:Stabreg}.  
For the continuous time fractional differential equations, several related results have been derived recently in \cite{wang2015dissipativity, wang2019dissipativity} under some structural assumptions.  For linear F-DDEs, by constructing a new generalized delay matrix function (of Mittag-Leffler type) and studying its related properties, the authors \cite{vcermak2016stability,vcermak2017fractional} not only gave the optimal decay rate of the solutions to F-DDEs accurately, but also a fully description of the stability region of F-DDEs. This new technique is further used in \cite{siegmund2020stability} to analyze nonlinear F-DDEs model. In \cite{tuan2020qualitative}, by using the linearization method combined with a weighted type norm, the authors presented various qualitative analyses including Mittag-Leffler stability for fractional systems with time varying delay. Note that Mittag-Leffler stability is a stronger concept than the usual asymptotic stability, it needs to finely describe the optimal algebraic decay rate of the solutions.

Compared with F-ODEs, the numerical stability analysis for F-DDEs will face several new challenges. Some authors have recently studied various numerical methods of F-DDEs \cite{dabiri2018numerical, kaslik2012analytical,maleki2019fractional,zayernouri2014spectral,
zhao180generalized}. In \cite{vcermak2020exact}, the authors provided the stability conditions for pure delay fractional differential equations for continuous and discrete solutions and proved that the fractional backward Euler method is not $\tau(0)$-stable.  This negative result is quite different from the classical DDEs, in which backward Euler is $\tau(0)$-stable \cite{bellen2013numerical, guglielmi1998delay,huang2009delay}.  This shows the complexity of the numerical stability region of the F-DDEs from  certain aspects, and also shows that the numerical stability region of F-DDEs is $\alpha$-dependent 
However, the accurate description of numerical stability region of F-DDEs for general scalar test models is still open.

On the other hand, as far as we know, there are few results for the long-term qualitative analysis of the numerical solutions of time fractional-order equations, i.e., the numerical Mittag-Leffler stability.   In \cite{wang2019dissipativity}, a preliminary study \tcb{was conducted} on the F-DDEs through energy methods. However, this method relies heavily on the special structure of the coefficients of the numerical scheme, so that only low-order schemes can meet the requirements. In fact, we guess they must be lower than the second order, and they must be the $\mathcal{CM}$-preserving schemes which are introduced in our recently work \cite{LiWang2019}. It is worth pointing out here that long term behavior analysis of numerical solutions for time fractional-order equations cannot normally be obtained from Gr\"{o}nwall-type inequalities \cite{wang2020ML}. 
  In order to overcome the weakness of the energy method in \cite{wang2019dissipativity}, especially in order to establish a direct connection with the numerical stability region, we turn to the singularity analysis of generating function as a new tool in \cite{LiWang2019,wang2020ML}, through which the numerical Mittag-Leffler stability theory of F-ODEs  with or without perturbations near  equilibriums was established in \cite{wang2020ML}.
 
In this paper, we focus on the following scalar test model for F-DDEs with order $\alpha\in(0, 1)$
\begin{gather}\label{eq:F-DDEs}
\begin{split}
&\mathcal{D}_c^{\alpha}y(t) =ay(t)+by(t-\tau), \quad t>0,\\
&y(t)=\varphi(t),\quad -\tau \leq t\leq 0,
\end{split}
\end{gather}
where $a,b\in\mathbb{R}$, the delay $\tau>0$ is a fixed constant, $\mathcal{D}_c^{\alpha}y(t):=\frac{1}{\Gamma(1-\alpha)}\int_{0}^{t} \frac{y'(s)}{(t-s)^{\alpha}}ds$ stands for the Caputo fractional derivative and $\varphi$ is the initial function.
Although this model is relatively simple, it is the basis for understanding more complex nonlinear models, and a thorough understanding of this model may provide much insight into numerical stability of F-DDEs. Note that when $b=0$ and $a\in\mathbb{C}$, \eqref{eq:F-DDEs} is reduced to the fractional test model or linear Volterra integral equation model, and its theoretical or numerical stability has been studied a lot \cite{brunner2017volterra,bellen2013numerical}, among which the basic stability theorem of fractional linear multistep method or convolution quadrature was proved in  \cite{lubich1986stability}. When $a=0$ and $b\in\mathbb{C}$, \eqref{eq:F-DDEs} becomes the fractional pure delay differential  equation, for which the stability results are studied in \cite{vcermak2020exact}.

Our goal in this article is twofold. Firstly, we aim to give an accurate description of the stability region of the numerical solutions of F-DDEs based on the GL scheme, and further study the $\tau(0)$-stability of the numerical scheme. Secondly, we prove the Mittag-Leffler stability of the numerical solutions provided that the parameters fall into the numerical stability region by the singularity analysis of generating function. The main results can be summarized as follows.

\begin{theorem}[Reformulation of Theorem \ref{thm:numericalregion} and Theorem \ref{thm:noabsstability}]
Fix $k\in \mathbb{N}_+$, $h=\tau/k$ and consider the GL method for \eqref{eq:F-DDEs}. When $k=1$, the numerical stability region $\mathcal{S}_k$ in the $(a, b)$-plane lies in the region between $a+b=0$ and $a-b=(2/\tau)^{\alpha}$.
When $k\ge 2$, the numerical stability region $\mathcal{S}_k$ in the $(a, b)$-plane lies between $a+b=0$
and the curve $\Gamma_0$:
\begin{equation}\label{eq:main}
\begin{split}
\Gamma_0:\quad \left\{
\begin{split}
& a=2^{\alpha}h^{-\alpha}\sin^{\alpha} \left( \frac{h\theta}{2\tau} \right) \frac{\sin(\theta+\alpha(\pi/2-h\theta/(2\tau)))}
{\sin(\theta)},\\
& b=-2^{\alpha}h^{-\alpha}\sin^{\alpha} \left( \frac{h\theta}{2\tau} \right) \frac{\sin(\alpha\pi/2-\alpha  h\theta/(2\tau))}{\sin(\theta)},
\end{split}\right.
\quad\quad
\theta\in \left( \frac{1-\alpha}{1-\alpha/k}\pi, \pi \right).
\end{split}
\end{equation}

Moreover, for any $k\ge 1$, there is always a portion of the stability region $\mathcal{S}_*$ for \eqref{eq:F-DDEs} that is outside the numerical stability region $\mathcal{S}_k$. Consequently, the numerical method is never $\tau(0)$-stable for $\alpha\in (0, 1)$.
\end{theorem}

The Mittag-Leffler stability holds,
\begin{theorem}[Reformulation of Theorem \ref{thm:main}]
\label{thm:MLstab}
Let $\alpha\in(0,1)$, $a,b\in\mathbb{R}$,  $k\in\mathbb{N}^{+}$ and $h=\tau/k$. 
Then the numerical solutions for \eqref{eq:F-DDEs} based on the GL method is Mittag-Leffler stable if $(a, b)$ falls into the numerical stability region $\mathcal{S}_k$ described above, or 
\[
y_{n}\sim - \frac{ y_{0}} { \Gamma(1-\alpha) \left( a +b\right) } t_n^{-\alpha} =O(t_n^{-\alpha}),
~~n\to\infty.
\]
\end{theorem}

The rest of this article is organized as follows. In Section \ref{sec:Stabreg}, we review the related concepts of the stability region for continuous F-DDE model and the corresponding main results. In Section \ref{sec:blt}, in parallel with the continuous case, we first introduce the concepts of relative stability and stability region of numerical solutions. Then, we derive the characteristic polynomial for the numerical scheme based on Gr\"{u}nwald-Letnikov (GL) approximation for the Caputo fractional derivative by the discrete Laplacian transform and generating functions.  By means of the boundary locus technique, the exact numerical stability region is given in Section \ref{sec:numstab}. In particular, we compare the differences in the numerical stability region between F-DDEs and classical DDEs, and prove by analysis and drawing that F-DDEs is never $\tau(0)$-stable; see Section \ref{sec:Notau}. Further, we show that as long as the numerical solution is stable, it is also Mittag-Leffler stable in Section \ref{sec:mlstability} by the technique of singularity analysis of generating function, that is, the numerical solution has a polynomial decay rate for long time that is exactly the same as the continuous equation. 

\section{Stability regions of the linear scalar F-DDEs}
\label{sec:Stabreg}

In this section, we collect some definitions and stability properties
for the linear scalar test model \eqref{eq:F-DDEs}, which will be the foundation of our study in later sections. 
We first observe that if we redefine
\begin{gather}\label{eq:scaling}
\tilde{a}=a\tau^{\alpha},~~\tilde{b}=b\tau^{\alpha},~~\tilde{t}=t/\tau,
~~\tilde{y}(\tilde{t})=y(t),~~\tilde{\varphi}(\tilde{t})=\varphi(t),
\end{gather}
then the form of \eqref{eq:F-DDEs} stays unchanged with $\tau=1$.
Hence, dropping the tildes, we will study the test model \eqref{eq:F-DDEs} with $\tau=1$:
\begin{gather}\label{eq:FDDE1}
\begin{split}
&\mathcal{D}_c^{\alpha}y(t) =ay(t)+by(t-1), \quad t>0,\\
&y(t)=\varphi(t),\quad -1 \leq t\leq 0.
\end{split}
\end{gather}
As long as the properties of \eqref{eq:FDDE1} are obtained, the properties for general $\tau$ can be easily obtained using the relations \eqref{eq:scaling}.

We will mainly study the stability properties of the model presented in \eqref{eq:FDDE1}. For this purpose, we introduce the following definitions.
\begin{definition}\label{def:stabilitycontinuous}
Let $\alpha\in(0,1)$ and consider parameters $a,b\in\mathbb{R}$.
\begin{enumerate}[(1)]
\item The asymptotic stability region in the $(a,b)$-plane is defined by $\mathcal{S}_*=\{(a,b): y(t)\to 0~ \text{as}~ t\to +\infty\}$ for all initial function $\phi$.
\item The solution is called Mittag-Leffler stable if $\|y(t)\|\leq C_{\alpha} t^{-\alpha}~ \text{as}~ t\to +\infty$, where the constant $C_{\alpha}>0$ is independent of $t$.
\end{enumerate}
\end{definition}

Obviously, Mittag-Leffler stability is a stronger result than asymptotic stability. Asymptotic stability only requires that the solutions tend to zero, but does not describe the decay rate. Mittag-Leffler stability requires that the decay rate of the solutions being algebraic, which is a typical feature of time fractional differential equations. The fundamental stability result for \eqref{eq:FDDE1} (or equivalently \eqref{eq:F-DDEs}) has been proved in \cite{vcermak2016stability,vcermak2017fractional}.

\begin{lemma}[{\cite[Theorem 4,5]{vcermak2017fractional}}] \label{lem:Laptrans}
Let $\alpha\in(0,1)$, and $a,b\in\mathbb{R}$. The zero solution of \eqref{eq:FDDE1} is asymptotically stable if and only if $(a, b)$ is an interior point of the region $\mathcal{S}_*$, bounded by the line $a+b=0$ from above and by the follwoing parametric curve $\Gamma$ from below
\begin{equation}\label{eq:truestab}
\Gamma:\quad a=   \frac{ \theta^{\alpha} \sin{\left(\theta+\frac{\alpha\pi}{2}\right)} } { \sin(\theta)}, 
\quad b=- \frac{ \theta^{\alpha} \sin{\left(\frac{\alpha\pi}{2}\right)} } { \sin(\theta)}, 
\quad \theta \in \left((1-\alpha)\pi,  \pi\right).
\end{equation}
In the asymptotic stability region, the solution is also Mittag-Leffler stable i.e., $\|y(t)\|\leq C_{\alpha} t^{-\alpha}$ as $t\to+\infty$.
\end{lemma}

The stability region $\mathcal{S}_*$ has a vertex $P=P(a, b)$, where 
$a=-b=\frac{[((1-\alpha)\pi]^{\alpha}\sin(\frac{\alpha\pi}{2})}{\sin(\alpha\pi)}>0.$
The region can be decomposed into two parts $\mathcal{S}_*=\mathcal{R}_1\cup \mathcal{R}_2$, where
\begin{equation}\label{eq:D1D2}
\begin{split}
(1) \quad \mathcal{R}_1: &=\{(a, b): a\le b<-a, a\le 0 \},\\
(2) \quad \mathcal{R}_{2}: &=|a|+b<0 ~\text{and}~ 1<\frac{(1-\alpha)\pi/2+\arccos[(-a/b)\sin(\alpha\pi/2)]}{ \left[ a\cos(\alpha\pi/2)+\sqrt{b^2-a^2\sin^2(\alpha\pi/2)} \right]^{1/\alpha}}.
\end{split}
\end{equation}


The stability region $\mathcal{S}_*=\mathcal{R}_1\cup \mathcal{R}_2$  in the $(a,b)$-plane characterized by Lemma \ref{lem:Laptrans} for $\alpha=0.5$ and $\tau=1$ can be found in Fig. \ref{Lineposi}.

\begin{remark}
If we want the results for general $\tau$, we may use the scaling \eqref{eq:scaling} and obtain, for example, the lower boundary curve as
\[
a=   \frac{ \tau^{-\alpha}\theta^{\alpha} \sin{\left(\theta+\frac{\alpha\pi}{2}\right)} } { \sin(\theta)}, 
\quad b=- \frac{ \tau^{-\alpha}\theta^{\alpha} \sin{\left(\frac{\alpha\pi}{2}\right)} } { \sin(\theta)}, 
\qquad \theta \in \left((1-\alpha)\pi,  \pi\right).
\]
Redefining $\tilde{\theta}=\tau^{-1}\theta$, one has
\[
a=   \frac{\tilde{\theta}^{\alpha} \sin{\left(\tau\tilde{\theta}+\frac{\alpha\pi}{2}\right)} } { \sin(\tau\tilde{\theta})}, 
\quad b=- \frac{\tilde{\theta}^{\alpha} \sin{\left(\frac{\alpha\pi}{2}\right)} } { \sin(\tau\tilde{\theta})}, 
\qquad \tilde{\theta} \in \left(\frac{(1-\alpha)\pi}{\tau},  \frac{\pi}{\tau}\right),
\]
which is the same as in \cite{vcermak2017fractional}.
\end{remark}

\begin{remark}
It is worth notting that the stability results for F-DDEs presented in Lemma \ref{lem:Laptrans} are $\alpha$-robust. That is, when $\alpha\to 1^{-}$, 
the stability region for F-DDEs $\mathcal{S}_*=\mathcal{S}_*(\alpha)$ can be reduced to the stability region for the corresponding integer DDEs \cite{vcermak2017fractional}. 
\end{remark}

Below, we describe a property of the lower boundary curve given by \eqref{eq:truestab}, 
which will be quite useful when we discuss the numerical stability regions later. 
\begin{lemma}\label{lmm:continuouscurve}
Consider the the lower boundary curve $\Gamma$ defined in \eqref{eq:truestab}. Along $\Gamma$, the equantity $\lambda(\theta):=a(\theta)-b(\theta)$ decreases as $\theta$ increases. Geometrically, this means that the straightline $a-b=\lambda(\theta)$ that goes through the point $(a(\theta), b(\theta))$ becomes higher and higher in the $(a, b)$ plane as $\theta$ increases.
\end{lemma}

\begin{proof}
Elementary properties of trigonometric functions imply that
\[
\lambda(\theta)=\frac{\theta^{\alpha}\sin(\theta/2+\alpha\pi/2)}{\sin(\theta/2)},
\]
 where $\theta \in \left((1-\alpha)\pi,  \pi\right)$. 
Now, we show that $\lambda(\theta)>0$ is a strictly decreasing function. To do this, we compute that
\[
\frac{d}{d\theta}\ln(\lambda(\theta))= \frac{\alpha}{\theta}-\frac{\sin(\alpha\pi/2)}{2\sin(\theta/2)\sin(\theta/2+\alpha\pi/2)}.
\]
To show the right hand side is negative, it suffices to so that $m(\theta):=\theta-2\alpha\frac{\sin(\theta/2)\sin(\theta/2+\alpha\pi/2)}{\sin(\alpha\pi/2)}>0$.
Clearly, $m(0)=0$ and $\frac{d}{d\theta}m(\theta)=1-\alpha \frac{\sin(\theta+\alpha\pi/2)}{\sin(\alpha\pi/2)}\ge 1-\frac{\alpha}{\sin(\alpha\pi/2)}>0$ for $\alpha\in (0, 1)$.
This means that $m(\theta)>0$ for $\theta\in (0, \pi)$. This completes the proof. 
\end{proof}

\section{Boundary locus technique for the numerical stability region}
\label{sec:blt}

In this section, we consider the boundary locus of the stability region for the F-DDEs model problem with the first order approximation for Caputo derivative by the well-known Gr\"{u}nwald-Letnikov formula \cite{diethelm10}. The GL scheme is  also the convolution quadrature generated by backward Euler scheme \cite{lubich1986stability}.

Consider \eqref{eq:FDDE1} and the step size $h=1/k$ for for some $k\in \mathbb{N}^{+}$. The mesh grid is given by $t_n=nh$. We will denote $y_n$ or $f_n$ to be the value defined at $t_n$. The numerical scheme for the F-DDE \eqref{eq:FDDE1} based on the GL scheme can be written as 
\begin{gather}\label{eq:FBDF1}
\mathcal{D}_h^{\alpha}y_n:=\frac{1}{h^{\alpha}}\sum_{j=0}^{n} \omega_{n-j} (y_j-y_0) = a y_{n}+ by_{n-k}, \quad n\geq 1, 
\end{gather}
where $(1-z)^{\alpha}=\sum_{j=0}^{\infty}\omega_jz^j$, and we have the explicit formula $\omega_{0}=1, \omega_{j}=\left(1-\frac{\alpha+1}{j} \right)\omega_{j-1}$ for $j\geq 1$. For the numerical solutions to exist, we require
\[
a\neq h^{-\alpha}=k^{\alpha}.
\]
In this paper, the branch cut of function $w\mapsto w^{\alpha}$
is taken to be $(-\infty, 0)$. In other words, $(re^{i\theta})^{\alpha}=r^{\alpha}e^{i\alpha\theta}$ for $\theta\in (-\pi, \pi]$.
Clearly, $y_{-k}$ is not used. When $k=1$, the data for $t<0$ are not used.  It is also worth noting that the GL formula is exactly the numerical scheme given in \cite{vcermak2020exact}.

We remark that the scaling \eqref{eq:scaling} is consistent with the numerical method \eqref{eq:FBDF1} as well due to the linearity of the GL scheme. 
In fact, one just needs to define $h=\tau/k$, $\tilde{h}=h/\tau$ for general $\tau>0$. Hence, studying \eqref{eq:FBDF1} is sufficient. As soon as the results here are obtained, the results for general $\tau$ can be obtained by simple scaling.

Parallel to Definition \ref{def:stabilitycontinuous}, we define some notions of stability for the numerical solution, following \cite{guglielmi1998delay}.
\begin{definition}
\begin{enumerate}[(1)]
\item 
For a given positive integer $k\geq 1$, define the numerical stability region $\mathcal{S}_{k}$ to be the set of pairs $(a,b)$ such that the numerical solutions with constant step size $h=1/k$ satisfy that $y_{n}\to 0~ \text{as}~ n\to +\infty$ for all initial function $\varphi$. 

\item The $\tau(0)$-stability region of a numerical method for F-DDEs is defined by  
\begin{gather}
\mathcal{S}_{\tau(0)}=\bigcap_{k\geq 1} \mathcal{S}_{k}.
\end{gather}
The numerical method is called $\tau(0)$-stable if $\mathcal{S}_{*}\subset \mathcal{S}_{\tau(0)}$.

\item If $|y_{n}|\leq C_{\alpha}t_n^{-\alpha}$ as  $n\to +\infty$, we call the numerical solutions are Mittag-Leffler stable.
\end{enumerate}
\end{definition}

The boundary locus technique is used to determine the boundary of the stability region $\mathcal{S}_k$ via the discrete Laplace transform (or equivalently the generating function) of the numerical solution. It has been widely used to determine the accurate stability region for various continuous and discrete test models \cite{bellen2013numerical, vcermak2017fractional, vcermak2020exact}.

\subsection{The discrete Laplace transform and the generating functions}

Consider the discrete function sequence $f=(f_0, f_1,  f_2,...)$ defined on the mesh points $t_n$.  Let $[\cdot]_{n}$ be the $n$-th entry of a sequence (i.e., $[f]_{n}=f_{n}$). The discrete Laplace transform is defined to be
\begin{gather}\label{eq:FLaptran}
\mathscr{L}_{h}\{f\}(s) := h\sum _{j=0}^{\infty} f_{j} (1-hs)^{j}, \quad s\in\mathbb{C},
\end{gather}
where the complex number $s$ is taken such that the the above series is convergent. 
 If the serie converges at some $s\neq h^{-1}$, then there exists $r>0$ such that it will also be convergent on the disk $D(h^{-1}, r):= \left\{z\in\mathbb{C}: |s-h^{-1}|<r \right\}$. 

\begin{remark}
Note that the definition \eqref{eq:FLaptran} is different from the discrete Laplace transform in \cite{vcermak2020exact}
because their sequence starts with $j=1$. We choose the definition \eqref{eq:FLaptran} mainly because it is consistent 
with the convolution we introduce later.
\end{remark}

The main advantage of the discrete Laplace transform defined in \eqref{eq:FLaptran} is that it has quite nice properties similar to the continuous Laplace transform, especially for the discrete fractional operator and delay function. 
To illustrate these, we consider the weighted discrete convolution $(\cdot * \cdot)_h$ for functions defined on the grid $t_n$:
\[
[(f*g)_h]_n:=h\sum_{j=0}^n f_j g_{n-j}.
\]
Note that the discrete convolution here is also different from that in \cite{vcermak2020exact} simply because starting with 
$j=0$ seems more convenient. More over, we also consider the difference operator 
\begin{gather}
\nabla_h y_n=\begin{cases}
0, & n=0,\\
h^{-1}(y_n-y_{n-1}), & n\ge 1.
\end{cases}
\end{gather}

The following important properties presented in  Lemma \ref{lem:Lappro} and Lemma \ref{lem:diswiener} are straightforward from the definitions:
\begin{lemma}
\label{lem:Lappro}
Let the functions $f$ and $g$ such that $\mathscr{L}_{h}\{f\}$ and $ \mathscr{L}_{h}\{g\}$ converge on $D(h^{-1}, r_{f})$ and  $D(h^{-1}, r_{g})$, respectively. Then one has that 
\begin{enumerate}[(i)]
\item  $\mathscr{L}_{h}\{ (f*g)_h \}(s)=\mathscr{L}_{h}\{f \}(s)\cdot \mathscr{L}_{h}\{g \}(s)$ on $D(h^{-1}, r)$, where $r=\min\{r_f, r_g\}.$

\item $\mathscr{L}_{h}\{\nabla_h f \}(s)=s\mathscr{L}_{h}\{f \}(s)- f_0$ on $D(h^{-1}, r)$.

\item $\mathscr{L}_{h}\{ \mathcal{D}_h^{\alpha}f \}(s)=s^{\alpha}\mathscr{L}_{h}\{f \}(s)-s^{\alpha-1}f_{0}$.

\item $\mathscr{L}_{h}\{ f_{d_{k}} \}(z)= (1-hs)^{k}\mathscr{L}_{h}\{f \}(s) + h\sum_{j=1}^{k}f_{-j}(1-hs)^{k-j}$, where the $k$-step delay function given by $(f_{d_{k}})_{n}=f_{n-k}$ for $n\geq 0$.
\end{enumerate}
\end{lemma}

Note that the third property can be derived using the fact that
$\mathcal{D}_h^{\alpha}y_n= h^{-1-\alpha} [(\omega* (y-y_0))_h]_n$,
and the fact that $\mathscr{L}_{h}(\omega)=h^{1+\alpha} s^{\alpha}$, where $\omega=(\omega_0, \omega_1, \omega_2,...)$ and the coefficients are given in \eqref{eq:FBDF1}.
Note that though the definition now starts from $j=0$, the properties almost stay the same as those in \cite{vcermak2017fractional, vcermak2020exact} (the only difference happens for the Laplace transform of the delay functions).

The following important results from \cite{vcermak2017fractional, vcermak2020exact} characterize the asymptotic behavior for the discrete function in convergent region or its boundary by its discrete Laplacian transform.
We remark that though the definitions of the discrete Laplace transforms here are different from those in \cite{vcermak2017fractional, vcermak2020exact}, the properties here stay unchanged because the asymptotic behaviors will not change under reindexing.

\begin{lemma}\label{lem:diswiener}
(i) Assume that the function $f$ such that $\mathscr{L}_{h}\{f\}$  converge on $D(h^{-1}, r)$ with $r>0$.  If $r>h^{-1}$, then $f\in\ell^{1}$. Otherwise, if $r<h^{-1}$, then $\lim_{n\to\infty} |f_n|=\infty$. 

(ii) Let $g\in\ell^1$. Then there exists $f\in\ell^1$ such that $\mathscr{L}_{h}\{f \}(z)\cdot \mathscr{L}_{h}\{g \}(z)=1$ if and only if 
$$
\inf \left\{ |\mathscr{L}_{h}\{g \}(z)|: z\in \mathrm{cl} \left( D(h^{-1}, h^{-1}) \right) \right\}>0.
$$
where $\mathrm{cl}$ means the ``closure".
\end{lemma}

A related function is the generating function of a sequence $f=(f_0, f_1, \ldots)$ defined by 
\begin{gather}\label{eq:gf}
\mathscr{F}_f(z)=\sum_{n=0}^{\infty} f_n z^n, \quad z\in\mathbb{C}. 
\end{gather}
Clearly, the discrete Laplace transform is related to the generating function by
\begin{gather}\label{eq:laplacetogenerating}
\mathscr{L}_{h}\{f\}(s)= h\mathscr{F}_f(1-hs).
\end{gather}

The generating function forgets the underlying grid $t_n$ and is applicable for any given sequence.
Corresponding to the generating function, one may also define the discrete convolution that is unrelated to the underlying grid as well. If $f$ and $g$ are two scalar sequences with $f_{n}, g_{n}\in\mathbb{C}^{1}$, we define the discrete convolution 
\[
f*g=w,\quad w_{n}=\sum_{j=0}^{n}f_{n-j}g_{j}.
\]
It is noted that the only difference from $(\cdot * \cdot)_h$ is an $h$ factor. When $f*g=\delta_{d}$, where $\delta_{d}:=(1, 0, 0,...)$ is the convolutional identity, we call the sequence $f$ is invertible, and write the inverse $g=f^{(-1)}$.  Note that a sequence $f$ is invertible if and only if $f_{0}\neq 0$.   It is straightforward to verify that $\mathscr{F}_{f*g}(z)=\mathscr{F}_f(z)\cdot \mathscr{F}_g(z)$. 

Clearly, the discrete Laplace transform and the generating functions are the same thing in different disguise. We introduce the discrete Laplace transform simply because many of the numerical stability results in the literature are obtained by the discrete Laplace transform. The purpose to introduce generating functions is simply because many of the Tauberian type results \cite{widder41,FS90} are given in terms of generating functions (see section \ref{sec:mlstability} for more details). Below, we explain briefly how one can use the discrete Laplace transform or the generating functions to characterize the numerical stability of the discrete solutions.

Applying the discrete fractional Laplace transform to the numerical scheme \eqref{eq:FBDF1} (when $n=0$, the discrete Caputo derivative is zero), we get the Laplace transformation of the numerical solution sequence $y=(y_0, y_1, y_2,...)$ as that 
\begin{equation}\label{eq:FLaptrany}
\mathscr{L}_{h}\{y\}(s) =\left( s^{\alpha}-a-b(1-hs)^{k} \right)^{-1} \cdot \left( s^{\alpha-1}y_0-ah y_0+bh\sum _{j=1}^{k-1} y_{-j} (1-hs)^{k-j}\right).
\end{equation}
Introduce 
\begin{gather}\label{eq:characteristicfunc}
Q(s):= s^{\alpha}-a-b(1-hs)^{k}
\end{gather} 
with $h>0$ and $k\in\mathbb{N}^{+}$, which is defined to be the characteristic polynomial of the numerical scheme \eqref{eq:FBDF1}.

Now, we use zeros of $Q(s)$ and Lemma \ref{lem:diswiener} to investigate whether $|y_n|$ goes to zero or not.
The issue is that the zero of the denominator and the numerator in \eqref{eq:FLaptrany} may cancel.  
Some careful checking leads to the following key observations.  Part of them have essentially been proved in \cite{vcermak2020exact}.

\begin{proposition}\label{pro:asympbehavior}
Suppose that $a\neq h^{-\alpha}$.
\begin{enumerate}[(i)]
\item If $Q(s)$ has no root in $\mathrm{cl}(D(h^{-1}, h^{-1}))$, then $|y_n|\to 0$.
\item If $a+b\neq 0$ or $k\ge 2$, when $Q(s)$ has a root inside the open disk $D(h^{-1}, h^{-1})$, then there are some initial data $\{y_j: j\le 0\}$ that lead to $|y_n|\to \infty$. 
\item If $a+b=0$ and $k=1$, then $Q(s)=s^{\alpha}-as$ has two roots $s=0$ and $s=a^{-1/(1-\alpha)}$. In this case, $y_n=y_0\not\to 0$ in the numerical solutions. If $Q(s)$ has a root in the open disk $D(h^{-1}, h^{-1})=D(1, 1)$ (i.e., $a>2^{\alpha-1}$), then for parameters $(a', b')$ near the point $(a,b)$ with $a'+b'\neq 0$, $|y_n|\to\infty$.  Otherwise when $a\le 2^{\alpha-1}$, there is $(a', b')$ that is arbitrarily close to $(a, b)$ such that the numerical solution is stable.
\end{enumerate}
\end{proposition}

\begin{proof}
For (i), one first observe that the sequence whose discrete Laplace transform is $1/Q(s)$, which is in $\ell^1$ provided that
$Q(s)$ has no zero on $\mathrm{cl}(D(h^{-1}, h^{-1}))$. This is in fact a corollary of Lemma \ref{lem:diswiener} (ii).
Moreover, the sequence whose discrete Laplace transform is $s^{\alpha-1}y_0-ah y_0+bh\sum _{j=1}^{k-1} y_{-j} (1-hs)^{k-j}$ is in $\ell^p$ for some $p>1$. In fact, only the term $y_0 s^{\alpha-1}$ affects the asymptotic behavior, and the corresponding sequence decays like $n^{-\alpha}$. Then, $y_n$, as the convolution of these two, goes to zero.
See the proof of \cite[Theorem 3]{vcermak2020exact} for related proofs.

For (ii), when $k\ge 2$ and $b\neq 0$, there are always some initial data $\{y_j: j\le 0\}$ that can make the numerator of \eqref{eq:FLaptrany}
nonzero on the whole $D(h^{-1}, h^{-1})$. Then, a zero of the denominator in $D(h^{-1}, h^{-1})$ will lead a singularity
inside the disk. Lemma \ref{lem:diswiener} (i) implies then that $|y_n|\to \infty$.
Now suppose that $k\ge 2$ and $b=0$. Then, 
\[
\mathscr{L}_{h}\{y\}(s) =\frac{s^{\alpha-1}-ah}{s^{\alpha}-a } y_0.
\]
Since $a\neq h^{-\alpha}$, the numerator cannot have the same zero as the denominator (we recall that $s^{\beta}=r^{\beta}e^{i\beta\theta}$ for $s=re^{i\theta+2im\pi}$ $m\in \mathbb{Z}$, $\theta\in (-\pi, \pi]$).
In the case $k=1$ (and thus $h=1$),
\[
\mathscr{L}_{h}\{y\}(s) =\frac{s^{\alpha-1}-a}{s^{\alpha}-a-b(1-s)}y_0.
\]
The numerator has a zero at $s^*=a^{-\frac{1}{1-\alpha}}$.
At this point, the denominator is $a^{-\frac{\alpha}{1-\alpha}}-a-b+ba^{-\frac{1}{1-\alpha}}
=(a+b)(a^{-1/(1-\alpha)}-1)$. Since $a\neq h^{-\alpha}=1$, this is zero only if $a+b=0$. Hence, in the case of (ii), the zero of the denominator and the numerator cannot be the same. Lemma \ref{lem:diswiener} (i) gives the result.

Consider (iii), one has
\[
\mathscr{L}_{h}\{y\}(s) =\frac{s^{\alpha-1}-a}{s^{\alpha}-as}y_0.
\]
The first part of the claim is trivial. Consider $a>2^{\alpha-1}$. Then, for any $(a', b')$ that is sufficiently close
to the point $(a, b)$ with $a'+b'\neq 0$, $Q(s)$ has a zero strictly inside the disk $D(1, 1)$, while the argument for (ii)
implies that the numerator is nonzero at the same point. Hence, $|y_n|\to\infty$.
Otherwise, if $a\le 2^{\alpha-1}$, we can find $(a', b')$ that is sufficiently close to the point $(a, b)$
such that $Q(s)$ has no roots in $\mathrm{cl}(D(1, 1))$. Then, for this case $|y_n|\to 0$. 
\end{proof}

By the discussion above, we can find the parameters $(a, b)$ such that $Q(s)$ has a root on the circle $\partial D(h^{-1}, h^{-1})$. This will be the boundary of the stability region. (This is true even for the case $a+b=0$ and $k=1$. In fact, by Proposition \ref{pro:asympbehavior} (iii), the portion of $a+b=0$ with $a\le 2^{\alpha-1}$ belongs to the boundary of the numerical stability region).

One may use the generating function as well to conclude the same boundary of the stability region. In fact,
Let $f_{n}=ay_n+by_{n-k}$. Then the equation \eqref{eq:FBDF1} can be written as 
\begin{gather}\label{eq:FBDFrd}
\frac{1}{h^{\alpha}}\sum_{j=0}^{n} \omega_{n-j} (y_n-y_0) = f_{n}-f_0\delta_{n,0}, \quad n\geq 0, 
\end{gather}
where $\delta_{i,j}=1$ if $i=j$ and $\delta_{i,j}=0$ if $i\neq j$ is the usual Kronecker function. 
Taking $\mu=\omega^{(-1)}$, with the generating function $F_{\mu}(z)=(1-z)^{-\alpha}$ for the Gr\"{u}nwald-Letnikov scheme, one then obtains (see more details in \cite{LiWang2019,wang2020ML})
\begin{equation}\label{eq:Fy0}
\begin{split}
\mathscr{F}_y(z)-y_{0}(1-z)^{-1}=h^{\alpha} \left( \mathscr{F}_\mu(z)\cdot \mathscr{F}_f(z)-f_{0}\mathscr{F}_{\mu}(z) \right).
\end{split}
\end{equation}
In view of $f_{n}=ay_n+by_{n-k}$ for any $n\geq 0$ so that $\mathscr{F}_f(z)=a \mathscr{F}_y(z)+b  \left( \sum_{\ell=0}^{k-1}y_{\ell-k}z^{\ell}+z^{k}\mathscr{F}_y(z) \right)$.
Substituting this expression into \eqref{eq:Fy0} yields that 
\begin{equation}\label{eq:Fy}
\begin{split}
\mathscr{F}_y(z) &=\left( 1-h^{\alpha} \left( a +bz^{k} \right) \mathscr{F}_\mu(z) \right)^{-1} \cdot \left( y_{0}(1-z)^{-1}+h^{\alpha} \left( bg(z)-ay_0 \right)\mathscr{F}_{\mu}(z)  \right)\\
&=\left( (1-z)^{\alpha}-h^{\alpha} \left( a +bz^{k} \right)  \right)^{-1} \cdot \left( y_{0}(1-z)^{\alpha-1}+h^{\alpha} \left( bg(z)-ay_0 \right)\right)
\end{split}
\end{equation}
where $g(z)=\sum_{\ell=1}^{k-1}y_{\ell-k}z^{\ell}$.
Clearly, \eqref{eq:laplacetogenerating} is satisfied. 
The asymptotic behavior of $y_n$ can also be studied using the properties of the generating function
on some neighborhood of unit disc $D(0,1)$, which has been widely studied in \cite{FS90}. This will be useful later in section \ref{sec:mlstability}. For the boundary locus, one needs to find the parameters $(a,b)$ such that $1-h^{\alpha} \left( a +bz^{k} \right) \mathscr{F}_\mu(z)$ has zeros on the unit circle, agreeing with the claims in Proposition \ref{pro:asympbehavior}.

\subsection{Boundary locus of the numerical stability region}

We investigate the boundary locus of the numerical stability region $\mathcal{S}_k$ by finding the parameters $(a, b)$
such that $Q(s)$ has a root on the circle $\partial D(h^{-1}, h^{-1})$. The smallest such region enclosed by these parameters will be the stability region $\mathcal{S}_k$, which obviously will include $\{(a, b):  a\in (-\infty, 0), b=0\}$.

\begin{lemma}\label{lmm:charequ}
Let $\alpha\in(0,1)$, $a,b\in\mathbb{R}$, $\tau>0$ and $k\in\mathbb{N}^{+}$. 
\begin{enumerate}[(i)]
\item If $a+b\geq 0$, then the equation $Q(z)=0$ has at least one nonnegetive real root on the disc $D(h^{-1}, h^{-1})$. 
\item If $z$ is a root of $Q(z)$, then its complex conjugate $\bar{z}$ is also a root of $Q(z)$, i.e.,  $Q(\bar{z})=0$.
\end{enumerate}
\end{lemma}
\begin{proof}
(i) Note that if $a+b\geq 0$, then $Q(0)=-a-b\leq 0$. 
It follows from the the assumption $a\neq h^{-\alpha}$ that $h^\alpha a<1$. On the other hand, we have that 
$Q(h^{-1})=h^{-\alpha}-a>0$. Hence, we know that $Q(z)=0$ has at least one nonnegetive real root. Further, $Q(z)=0$ has the zero root if $a+b=0$ and  
$Q(z)=0$ has a positive real root if $a+b>0$.

(ii) The second is obvious since all the parameters involved are real.
\end{proof}

Consider that  $s\in \partial D(h^{-1}, h^{-1})$, which can be parameterized as $s=2h^{-1}\cos(\phi)e^{i\phi}=h^{-1}(1+e^{i2\phi})$, where $\phi\in[-\pi/2, \pi/2]$. Hence, $1-hs=-e^{i2\phi}$. Then,
\begin{equation}\label{eq:Qequ}
\begin{split}
Q(s)= \tilde{Q}(\phi)=2^{\alpha} h^{-\alpha} \cos^{\alpha}(\phi)\left( \cos(\alpha \phi)+i \sin(\alpha \phi)\right) -a-b(-1)^{k} \left( \cos(2k\phi)+i \sin(2k\phi)\right).
\end{split}
\end{equation}

By Lemma \ref{lmm:charequ} (ii), $\tilde{Q}(\phi)$ is an even function. Hence, we need to find parameters $(a, b)$ such that the following hold for some $\phi\in [0, \frac{\pi}{2}]$.
\begin{equation}\label{eq:boundaryloc}
\begin{split}
\begin{cases} 
&2^{\alpha} h^{-\alpha} \cos^{\alpha}(\phi) \cos(\alpha \phi)=a+b(-1)^{k} \cos(2k\phi),\\
&2^{\alpha} h^{-\alpha} \cos^{\alpha}(\phi) \sin(\alpha \phi) =b(-1)^{k} \sin(2k\phi).\\
\end{cases} 
\end{split}
\end{equation}

Let's discuss the equation \eqref{eq:boundaryloc} in three different cases.

{\bf{Case (I)}}: If $\phi=0$, then $a+(-1)^k b=(2/h)^{\alpha}$.
  
{\bf{Case (II)}}: If $\phi=\pi/2$, we get the curve $a+b=0$. 

{\bf{Case (III)}}: Otherwise, $\sin(2k\phi)$ can never be zero. We can solve the equation \eqref{eq:boundaryloc} to get that
\[
 a=\frac{2^{\alpha}h^{-\alpha}\cos^{\alpha}(\phi)\sin(2k\phi-\alpha\phi)}
{\sin(2k\phi)}, \quad
 b=(-1)^k2^{\alpha}h^{-\alpha}\cos^{\alpha}(\phi)\frac{\sin(\alpha\phi)}{\sin(2k\phi)}.
\]

In the third case, by setting
$\theta=k\pi-2k\phi\in [0,k\pi]$,
one has
\begin{equation}\label{eq:abformula}
 a=\frac{2^{\alpha}h^{-\alpha}\sin^{\alpha} \left(\frac{\theta}{2k} \right)\sin(\theta+\alpha(\pi/2-\theta/2k))}
{\sin(\theta)}, \quad
 b=-2^{\alpha}h^{-\alpha}\sin^{\alpha} \left( \frac{\theta}{2k} \right)\frac{\sin(\alpha\pi/2-\alpha\theta/(2k))}{\sin(\theta)}.
\end{equation}

These curves obtained are essentially the boundary locus. The smallest region enclosed by them that include $(-\infty, 0)$
will be the stability region $\mathcal{S}_k$. We perform the discussion in the next section.

\begin{remark}
Clearly if $a\le 0$, then
$|b|\ge \Re(b(-e^{i2\phi})^k)=-a+(2h^{-1})^{\alpha}\cos^{\alpha}(\phi)\cos(\alpha\phi)\ge -a.$
Hence,  $\mathcal{R}_1$ is totally contained in this region. This means the numerical solution is stable for the parameter pairs $(a, b)\in \mathcal{R}_1$. 
The crucial part is to study the numerical stability when the parameter pairs $(a, b)\in \mathcal{R}_2$. 
In order to characterize the corresponding number stability region when $(a, b)\in \mathcal{R}_2$ accurately,  we next discuss it in several cases.
\end{remark}

\section{Numerical stability regions}
\label{sec:numstab}

In this section, we perform detailed discussions on the curves found by the boundary locus technique to find an explicit formula for the numerical stability region $\mathcal{S}_k$.

\subsection{Discussion on the straight lines for $\phi=0$ in Case (I)}
\label{sec:caseI}

When $k$ is even, the line $a+(-1)^k b=(2/h)^{\alpha}=(2k)^{\alpha}$ is above $a+b=0$, which never affects the stability.
Now, we assume that $k$ is odd, which leads to the line $a-b=2^{\alpha}k^{\alpha}$ in the parameter $(a, b)$-plane. Let us examine the position of this line in relation to the continuous stability region $\mathcal{S}_*$.

\begin{lemma}\label{lmm:locuslines}
\begin{enumerate}[(i)]
\item  If $k=1$, for any $\alpha\in (0, 1)$, the line $a- b=2^{\alpha}$ lies above part of the stability region $\mathcal{S}_*$ for the continuous model \eqref{eq:FDDE1}. 

\item If $\alpha=1$ or $k=2p+1\ge 3$, the line $a- b=(2k)^{\alpha}$ lies below the continuous stability region $\mathcal{S}_*$.
\end{enumerate}
\end{lemma}

By Lemma \ref{lmm:continuouscurve}, one only needs to show $(2k)^{\alpha}\ge 2\frac{[(1-\alpha)\pi]^{\alpha}\sin(\alpha\pi/2)}{\sin(\alpha\pi)}
=\frac{[(1-\alpha)\pi]^{\alpha}}{\sin(\pi(1-\alpha)/2)}$ for the lines being under the stability region.
For $k=1$, one only needs to show that $2^{\alpha}<\frac{[(1-\alpha)\pi]^{\alpha}}{\sin(\pi(1-\alpha)/2)}$. The detailed calculation is an elementary calculus and we attach it in Appendix \ref{app:lines} for a reference.
Note that the line $a-b=2^{\alpha}$ for $k=1$ may not intersect with $\Gamma$: it can be above the curve totally. Here $\Gamma$ is the blow boundary curve of stability region $\mathcal{S}_*$ defined in \eqref{eq:truestab}.  
In fact, $\lim_{\theta\to\pi}\lambda(\theta)=\pi^{\alpha}\cos(\frac{\alpha\pi}{2})$. We find that when $\alpha$ is small enough (near zero), $\pi^{\alpha}\cos(\frac{\alpha\pi}{2})>2^{\alpha}$. Only when $\alpha$ is 
big enough, $\pi^{\alpha}\cos(\frac{\alpha\pi}{2})<2^{\alpha}$, in which case the line intersects with $\Gamma$. See Figure \ref{Lineposi2}.


Lemma \ref{lmm:locuslines} implies the numerical solution is not stable when $k=1$ for some parameters $(a, b)\in \mathcal{S}_*$. This means that the numerical solution can not be $\tau(0)$-stable for F-DDEs.
As a remark, when $\alpha\to 1$, we have
\[
\frac{[(1-\alpha)\pi]^{\alpha}}{\sin(\pi(1-\alpha)/2)}\to 2.
\]
Hence, for $\alpha=1$, the straight lines will lie below the stability region. This is consistent with the fact that the backward Euler method is  $\tau(0)$-stable for classical DDEs, whereas it is not for F-DDEs.
This is an agreement with the results obtained in the case of pure delay F-DDEs with $a=0$, a special case  studied in \cite{vcermak2020exact}.

\begin{figure}[!ht]
\begin{center}
$\begin{array}{cc}
 \includegraphics[scale=0.45]{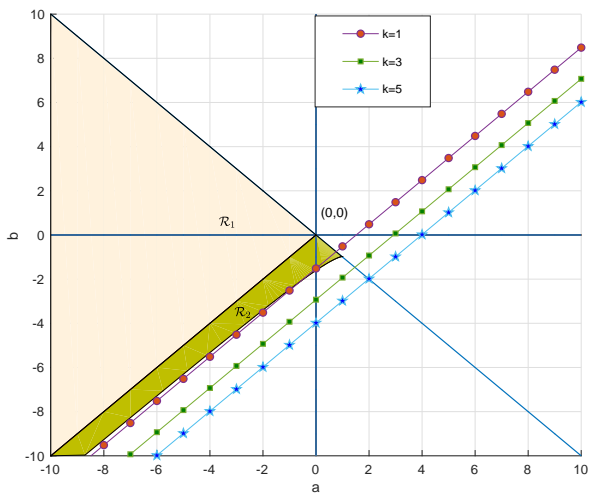}
 \end{array}$
\end{center}
\caption{Position relationship between the lines $a-b=(2k)^{\alpha}$ with $k=1, 3, 5$ and stability region  $\mathcal{S}_*$ for $\tau=1$ and $\alpha=0.6$.}
\label{Lineposi}
\end{figure}

To illustrate the results of this lemma, the position relationships between the lines $a-b=(2k)^{\alpha}$ and stability region $\mathcal{S}_*$ with different parameters are plotted in Fig. \ref{Lineposi}. It clearly shows that for $k=1$ the line intersects with $\mathcal{S}_*$ in the below boundary at one point and for $k=3, 5$ the lines lies below $\mathcal{S}_*$, as expected.

\begin{figure}[!ht]
\begin{center}
$\begin{array}{cc}
 \includegraphics[scale=0.7]{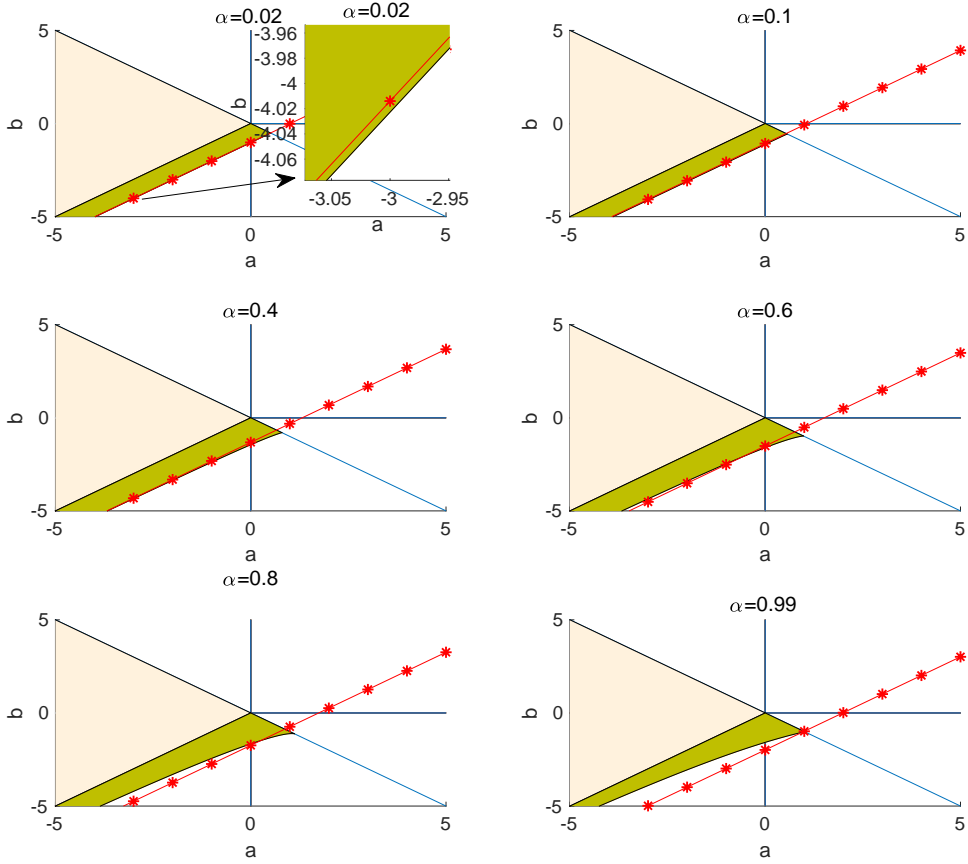}
 \end{array}$
\end{center}
\caption{Position relationship between the line $a-b=2^{\alpha}$ (red line with $*$) and stability region  $\mathcal{S}_*$ for $\tau=1$ and different order $\alpha=0.02, 0.1, 0.4, 0.6, 0.8$ and $0.99$.}
\label{Lineposi2}
\end{figure}

In order to illustrate the influence of order $\alpha$ on the numerical stability, and thus to distinguish the F-DDEs from the DDEs, the positional relationship between the line $a-b=2^{\alpha}$ and stability region $\mathcal{S}_*$ with different $\alpha$ is drawn in Fig. \ref{Lineposi2}. 
When the order $\alpha$ is very small, such as in the first the sub-figure for $\alpha=0.02$, the line $a-b=2^{\alpha}$ is totally above the curve $\Gamma$.
As the order $\alpha$ becomes larger, the line $a-b=2^{\alpha}$ intersects the curve $\Gamma$.
But as the order $\alpha$ gets larger and approaches $1^{-}$, the line slowly slides out of the stability region and finally does not intersect with $\mathcal{S}_*$. 

This shows that the numerical stability of GL for F-DDEs is very different from that of the integer DDEs, because of the change of the order $\alpha$, the numerical stability region can be reduced such that fractional backward Euler scheme becomes not $\tau(0)$-stable.

\subsection{Discussion on the parametrized curves in Case (III)}

As we have mentioned already, we only need to consider $\phi\in (0, \frac{\pi}{2})$ so that
$\theta=k\pi-2k\phi\in (0, k\pi)$ and $\sin\theta\neq 0$. 

We first consider $\theta\in ( (2m-1)\pi, 2m\pi)$ for some integer $m$, where $2m\le k$.
For $\theta$ in this range, $b>0$. Moreover, $\theta+\alpha(\pi/2-\theta/2k)\in \left[ (2m-1)\pi, 2m\pi+\frac{\alpha \pi}{2} \right)$ so $a$ goes from positive value to negative value. 
By noting the curve defined by the expression in \eqref{eq:abformula}, if it crosses $a+b=0$, it must happen for 
\begin{equation*}
\sin(\theta+\alpha(\pi/2-\theta/2k))-\sin(\alpha(\pi/2-\theta/2k))=0,
\end{equation*}
which means
\[
\theta+\alpha(\pi/2-\theta/2k)-\alpha(\pi/2-\theta/2k)=2m\pi.
\]
This is clearly impossible. Hence, the curve never intersects with $a+b=0$ and lies strictly above $a+b=0$. Hence, such curves cannot be the boundary of the numerical stability region (as $a+b=0$ is already the upper boundary).

Suppose $\theta\in (2m\pi, (2m+1)\pi)$ with $2m+1\le k$. Then, $\theta+\alpha(\pi/2-\theta/2k)\in \left( 2m\pi, \min \left\{(2m+1)\pi+\frac{\alpha \pi}{2}, k\pi \right\} \right)$.
In this case, $b<0$. The parameter $a$ either stays positve or changes from positive to negative.
This curve defined in in \eqref{eq:abformula} intersects $a+b=0$ at
$
\theta=\frac{(2m+1-\alpha)k\pi}{k-\alpha}\in (2m\pi, (2m+1)\pi].
$
Hence, we have a family of curves denoted by $\Gamma_m$ $\left( m=0,1,2,\cdots, \lfloor\frac{k-1}{2}\rfloor \right)$:
\begin{equation}\label{eq:gammcurve}
\Gamma_m:~~ \left\{
\begin{split}
& a=2^{\alpha}h^{-\alpha}\sin^{\alpha} \left( \frac{\theta}{2k} \right) \frac{ \sin(\theta+\alpha(\pi/2-\theta/2k))}
{\sin(\theta)},\\
& b=-2^{\alpha}h^{-\alpha}\sin^{\alpha} \left( \frac{\theta}{2k} \right)\frac{\sin(\alpha\pi/2-\alpha\theta/(2k))}{\sin(\theta)},
\end{split}
\quad\quad \theta\in \left[ \frac{(2m+1-\alpha)k\pi}{k-\alpha}, (2m+1)\pi \right).
\right.
\end{equation}

Now let's present our main results. That is, an accurate description of the numerically stable region.

\begin{theorem}
\label{thm:numericalregion}
Fix $k$ to be a positive integer. When $k=1$, the numerical stability region $\mathcal{S}_k$ in the $(a, b)$-plane lies in the region between $a+b=0$ and $a-b=2^{\alpha}$.
When $k\ge 2$, the numerical stability region $\mathcal{S}_k$ in the $(a, b)$-plane lies between $a+b=0$
and the curve $\Gamma_0$:
\begin{equation}\label{eq:main}
\begin{split}
\Gamma_0:\quad \left\{
\begin{split}
& a=2^{\alpha}k^{\alpha}\sin^{\alpha} \left( \frac{\theta}{2k} \right) \frac{\sin(\theta+\alpha(\pi/2-\theta/2k))}
{\sin(\theta)},\\
& b=-2^{\alpha}k^{\alpha}\sin^{\alpha} \left( \frac{\theta}{2k} \right) \frac{\sin(\alpha\pi/2-\alpha\theta/(2k))}{\sin(\theta)},
\end{split}\right.
\quad\quad
\theta\in \left( \frac{1-\alpha}{1-\alpha/k}\pi, \pi \right).
\end{split}
\end{equation}
\end{theorem}

\begin{figure}[!ht]
\begin{center}
$\begin{array}{cc}
 \includegraphics[scale=0.5]{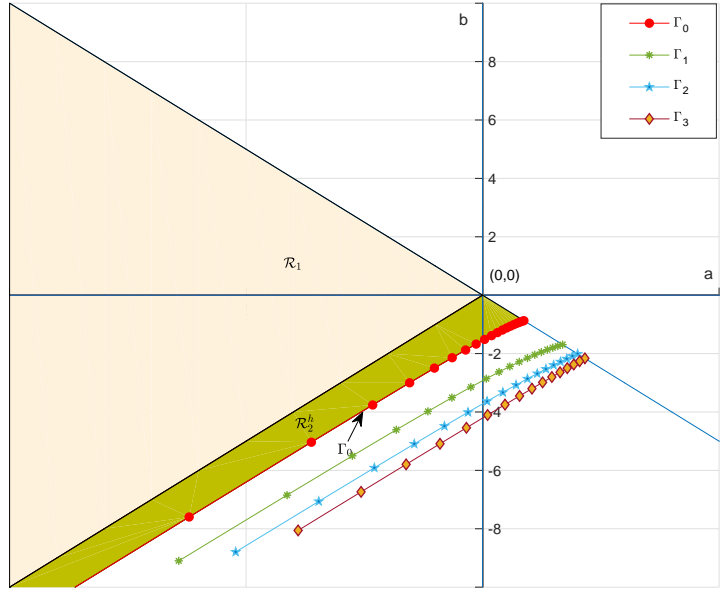}
 \end{array}$
\end{center}
\caption{The curves  $\Gamma_0, \Gamma_1, \Gamma_2$ and $\Gamma_3$ for $\tau=1$, $h=0.1, k=10$ and $\alpha=0.5$.}
\label{Numsta}
\end{figure}

The curves $\Gamma_0, \Gamma_1, \Gamma_2$ and $\Gamma_3$ for $\tau=1$, $h=0.1, k=10$ and $\alpha=0.8$ are plotted in Fig. \ref{Numsta}. It shows that 
the curve $\Gamma_0$ is above other curves $\Gamma_m$ for $m\geq 1$.

To prove Theorem \ref{thm:numericalregion}, we need some auxiliary lemmas.   

\begin{lemma}\label{lem:points}
The curves $\Gamma_m$ defined in \eqref{eq:gammcurve} intersects the line $a+b=0$
at the points with
\[
a_m=\frac{2^{\alpha}k^{\alpha}\cos^{\alpha} \left( \frac{(k-2m-1)\pi}{2(k-\alpha)} \right)}
{2\cos \left( \frac{\alpha\pi(k-2m-1)}{2(k-\alpha)} \right)}.
\]
They are increasing with respect to $m$ for $\alpha\in (0, 1)$ and for $\alpha=1$, $a_0<a_1=a_2=\cdots$.
\end{lemma}

\begin{proof}
Let $\phi:=\frac{(k-2m-1)\pi}{2(k-\alpha)}$. In view of the expression in \eqref{eq:gammcurve}, we find
$
a=-b=\frac{2^{\alpha}k^{\alpha}\cos^{\alpha}(\phi)}
{2\cos(\alpha\phi)}.
$
Taking  the derivative on $\ln (\cos^{\alpha}(\phi)/\cos(\alpha \phi))$, we find that this is decreasing on $[0, \pi/2]$.
Clearly, as $m$ increases, $\phi/2$ decreases on $(0, \pi/2)$.
Hence, $a$ increases. 
This means $m=0$ corresponds to the largest $a$ or highest intersection point.

The claim for $\alpha=1$ is obvious, which can be checked directly.
\end{proof}

For the family of curves $\Gamma_m$ defined in \eqref{eq:gammcurve}, $\Gamma_0$ plays an important role. We study the important properties 
for  $\Gamma_0$.

\begin{lemma}\label{lem:decrecruv}
For the curve $\Gamma_0$, the quantity
\[
\lambda_k(\theta):=a(\theta)-b(\theta)=2^{\alpha}k^{\alpha}\sin^{\alpha} \left( \frac{\theta}{2k} \right)
\frac{\sin \left(\theta/2+\alpha(\frac{\pi}{2}-\frac{\theta}{2k}) \right)}{\sin(\theta/2)},
\quad \theta\in \left( \frac{(1-\alpha)k\pi}{k-\alpha}, \pi \right),
\]
is decreasing for $k\ge 2$.
Moreover, for any $k\ge 2$, the straight line determined by $a-b=2^{\alpha}k^{\alpha}$ is below
the curve $\Gamma_0$.   
\end{lemma}
The proof of this lemma involves some nontrivial elementary calculation and is not central to the main result, so we defer it to Appendix \ref{app:Gamma0}.

The following lemma lays the foundation to studying the relative positions for the curves $\Gamma_m$.
The basic idea is to consider how the line $a-b=\mathrm{const}$ that goes through a specified point on $\Gamma_m$ corresponding to  $\theta$ changes when $\theta$ is increased by $2\pi$. Note that increasing $\theta$ by $2\pi$ corresponds to jumping from one point on $\Gamma_m$ to a point on  $\Gamma_{m+1}$ if $\theta+2\pi$ is a legal parameter for $\Gamma_{m+1}$.
\begin{lemma}\label{lem:lamd}
Consider again
$\lambda_k(\theta):=a(\theta)-b(\theta)$
for $(a, b)$ on the parametric curves $\Gamma_m$ defined in \eqref{eq:gammcurve}. Then, it holds that
\[
\lambda_k(\beta+2m\pi) >\lambda_k(\beta)
\]
 for $\beta\in \left[ \frac{(1-\alpha)k\pi}{k-\alpha}, \pi \right)$ and $\beta+2m\pi \in \left[ \frac{(2m+1-\alpha)k\pi}{k-\alpha}, (2m+1)\pi \right)$.
\end{lemma}

\begin{proof}
As in Lemma \ref{lem:decrecruv}, we find 
\[
\lambda_k(\beta+2m\pi)
=2^{\alpha}k^{\alpha}\sin^{\alpha} \left( \frac{\beta+2m\pi}{2k} \right)
\frac{\sin \left(\beta/2+m\pi+\alpha(\frac{\pi}{2}-\frac{\beta+2m\pi}{2k}) \right)}{\sin(\beta/2+m\pi)}
\]
Since $\frac{\sin \left( \beta/2+m\pi+\alpha \left( \frac{\pi}{2}-\frac{\beta+2m\pi}{2k} \right) \right)}{\sin(\beta/2+m\pi)}
=\frac{\sin\left( \beta/2+\alpha \left( \frac{\pi}{2}-\frac{\beta+2m\pi}{2k} \right) \right)} {\sin(\beta/2)}$, it suffices to show the following function is increasing:
\[
g(t)=\sin^{\alpha}(t)\frac{\sin \left( \beta/2+\alpha\frac{\pi}{2}-\alpha t \right)}{\sin(\beta/2)},
\quad t\in \left[ 0, \frac{\beta+2m\pi}{2k} \right].
\]
We find
$\frac{d}{dt}\ln g=\alpha[ \cot(t)-\cot(\beta/2+\alpha\pi/2-\alpha t)]$.
Note that $\beta/2<\pi/2$ and $t< \pi/2$, for this derivative to be positive, we need
$t<\beta/2+\alpha\pi/2-\alpha t$, or equivalently $t\le \frac{\beta/2+\alpha\pi/2}{1+\alpha}$. Hence, it suffices to show that  $\frac{\beta+2m\pi}{2k}\le \frac{\beta/2+\alpha\pi/2}{1+\alpha}$.

Noting that $\beta+2m\pi \ge \frac{(2m+1-\alpha)}{1-\alpha/k}\pi$, one then has
 $\beta \ge \frac{(1-\alpha)\pi+2m\alpha\pi/k}{1-\alpha/k}\ge \frac{2m\pi}{k-1}$.
In fact, this holds for both $\alpha=0$ and $\alpha=1$ ($2m+1\le k$). Because the expression here is monotone in $\alpha$, it holds for all $\alpha\in (0, 1)$.
Consequently, $\frac{\beta+2m\pi}{2k}\le \beta/2\le \frac{\beta/2+\alpha\pi/2}{1+\alpha}$ since $\beta< \pi/2$. Hence, the claim is verified.
\end{proof}

Lemma \ref{lem:lamd} tells us that the straight line $a-b=\lambda_m$ becomes lower if we move from $\Gamma_0$ to $\Gamma_m$ for the corresponding parameters.  Consider the points on $\Gamma_m$ corresponding to $\theta+2m\pi$ as $m$ changes. 
Making them continuous, we shall have a family of parametrized curves with fixed $k$:
\begin{equation}\label{eq:gammac}
\begin{split}
\gamma(t; \beta):=\left( \frac{2^{\alpha}k^{\alpha}\sin^{\alpha}(t)\sin(\beta+\alpha\pi/2-\alpha t)}{\sin(\beta)},~
\frac{-2^{\alpha}k^{\alpha}\sin^{\alpha}(t)\sin(\alpha\pi/2-\alpha t)} {\sin(\beta)} \right),
\end{split}
\end{equation}
where $\beta\in  \left( \frac{1-\alpha}{1-\alpha/k}\pi ,\pi\right)$.
Along these curves, one sees the points on $\Gamma_0$, $\Gamma_1, \cdots, \Gamma_m$ 
consecutively, and they become lower and lower. The intersections of family $\Gamma_m$ ($0\leq m\leq 3$) and some typical $\gamma$ curves are shown in Fig. \ref{gamma}. Note that $\gamma$ curve may only intersect some $\Gamma_m$'s. This means that
if we increase the angle by $2m\pi$, the parameter may fall out of $\left[ \frac{(2m+1-\alpha)k\pi}{k-\alpha}, (2m+1)\pi \right)$.

Lemma \ref{lem:lamd} is not enough to imply that $\Gamma_0$ is above $\Gamma_m$. We must need these curves to be well-ordered so that they may be used to show that $\Gamma_0$ is above $\Gamma_m$. The following observation fulfills this task.

\begin{figure}[!ht]
\begin{center}
$\begin{array}{cc}
 \includegraphics[scale=0.6]{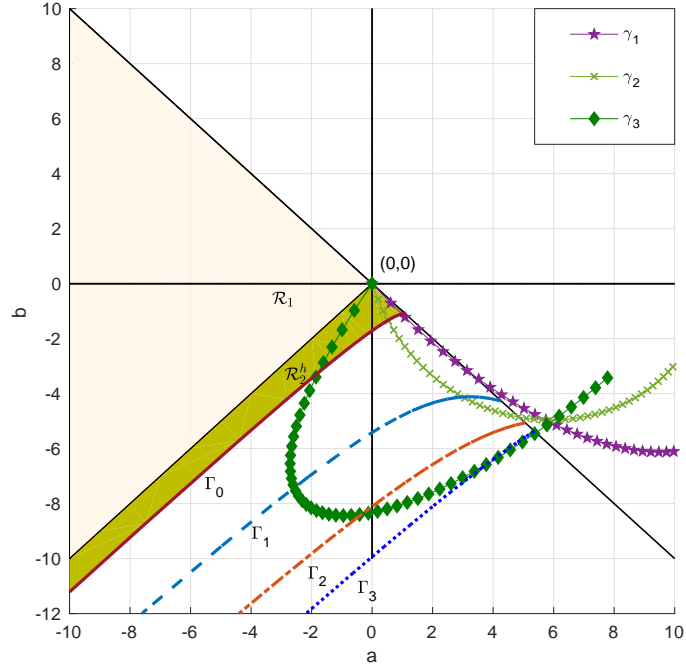}
 \end{array}$
\end{center}
\caption{Illustration of the curves $\gamma$. Here, $\gamma_1=\gamma(t; 0.3\pi), \gamma_2=\gamma(t; 0.5\pi) , \gamma_3=\gamma(t; 0.8\pi)$ for $\tau=1$, $h=0.1, k=10$ and $\alpha=0.8$.}
\label{gamma}
\end{figure}

\begin{lemma}\label{lem:observation}
For any $\beta_1\neq \beta_2$, $\gamma(t; \beta_1)$ does not intersect
$\gamma(t; \beta_2)$ in the region $|a|+b<0$.
\end{lemma}

\begin{proof}
Consider the line in the $(a, b)$-plane determined by
$
a/(-b)=\lambda\in (-1, 1).
$
We first observe that $a=-\lambda b$ intersects $\gamma(t; \beta)$
at at most one point.  In fact, at the intersection point,
\[
\frac{\sin(\beta+\alpha\pi/2-\alpha t)}{\sin(\alpha\pi/2-\alpha t)}
=\cos\beta+\sin(\beta)\cot(\alpha\pi/2-\alpha t)=\lambda.
\]
This implicitly defines a function $t=t(\beta)$ and we find
$\frac{\lambda-\cos\beta}{\sin\beta}=\frac{\cos(\alpha\pi/2-\alpha t)}{\sin(\alpha\pi/2-\alpha t)}$
which implies that 
$\alpha\frac{dt}{d\beta}=\sin^2 \left( \frac{\alpha \pi}{2}-\alpha t\right)
\frac{1-\lambda\cos\beta}{\sin^2\beta}>0.$
Hence, there is at most one $t$ that makes the equation hold so that there is at most one intersection. 

To show the curves $\gamma(t; \beta)$ for different $\beta$ values do not intersect in the region $|a|+b<0$,
it suffices to show that the intersections with $a=-\lambda b$ are monotone in $\beta$ for any $\lambda$ (the lines $a=-\lambda b$ can cover all the region $|a|+b<0$ as $\lambda$ varies).
For fixed $\lambda$, to show the intersections are monotone, we show $b$ is monotone in $\beta$, or
\[
h(\beta)=-b/(2^{\alpha}k^{\alpha})=\frac{\sin^{\alpha}(t)\cos(\alpha\pi/2-\alpha t)}{\lambda-\cos\beta}
\]
is strictly monotone.

The derivative of the function $\ln (h(\beta))$ reduces to check the sign of 
\[
\alpha\frac{dt}{d\beta} \cos\left(\frac{\alpha \pi}{2}-\alpha t-t\right)-\sin (t) \sin\left( \frac{\alpha \pi}{2}-\alpha t\right)
=\left( \alpha \frac{dt}{d\beta}-\frac{1}{2} \right) \cos \left( \frac{\alpha \pi}{2}-\alpha t-t\right)
+\frac{1}{2}\cos\left( \frac{\alpha \pi}{2}+(1-\alpha) t\right).
\]
Note that $\frac{\alpha \pi}{2}+t(1-\alpha)\in (0, \pi/2]$, so the second term is nonnegative.
If we can show $\alpha \frac{dt}{d\beta}> \frac{1}{2}$, then it is done.
By the relation above,
$\frac{\alpha\pi}{2}-\alpha t=\cot^{-1} \left( \frac{\lambda-\cos\beta}{\sin \beta} \right)$.
Direct computation shows
\[
-\alpha \frac{d t}{d\beta}=-\frac{1}{1+[(\lambda-\cos\beta)/\sin\beta]^2} \left( \frac{\lambda-\cos\beta}{\sin\beta}\right)'
=-\frac{1-\lambda\cos\beta}{1+\lambda^2-2\lambda\cos\beta}<-\frac{1}{2}.
\]
The last inequality is equivalent to
$2(1-\lambda\cos\beta)>1+\lambda^2-2\lambda\cos\beta$, which is clearly true.
\end{proof}

As shown in Fig. \ref{gamma}, the curves can intersection outside the region $|a|+b<0$.
However, since we only care about $\Gamma_m$ in the region $|a|+b<0$, the result above suffces. We can now prove the theorem.
\begin{proof}[Proof of Theorem \ref{thm:numericalregion}]
If $k=1$, there is only $\Gamma_0$ curve. 
Since $\theta+\alpha \left( \frac{\pi}{2}-\frac{\theta}{2k} \right)\in (0, \pi)$, $a>|b|>0$
The curve never intersects $a+b=0$. Hence, the curve $\Gamma_0$ cannot 
be the boundary of the numerical stability region. Only the straight line $a-b=2^{\alpha}$ is the lower boundary.

Now, assume $k\ge 2$. According to Lemma \ref{lem:decrecruv}, we only need to show that
$\Gamma_0$ is above $\Gamma_m$ for $m=2,\cdots, \lfloor\frac{k-1}{2}\rfloor$.
By Lemma \ref{lem:points}, the starting point is monotone, it reduces to checking that 
$\Gamma_m$ never intersects $\Gamma_0$ in $|a|+b<0$.

Suppose otherwise $\Gamma_m$ intersects $\Gamma_0$
at some point $(a_*, b_*)$ satisfying that $|a_*|+b_*<0$.
Then, this point corresponds to some $\theta_m\in \left( \frac{(2m+1-\alpha)\pi}{1-\alpha/k}, (2m+1)\pi \right)$ for $\Gamma_m$ and 
some $\beta_1\in [ \frac{(1-\alpha)k\pi}{k-\alpha}, \pi )$ for $\Gamma_0$.
It can be seen easily that $\beta_2:=\theta_m-2m\pi\in \left( \frac{1-\alpha}{1-\alpha/k}\pi, \pi \right)$.
By Lemma \ref{lem:lamd}, we have $\lambda(\theta_m)>\lambda(\beta_2)$.
This means that the following point
\[
\left( \bar{a}, \bar{b} \right)=
2^{\alpha}k^{\alpha} \left( \frac{\sin^{\alpha} \left( \frac{\beta_2}{2k}\right)\sin(\beta_2+\alpha(\pi/2-\beta_2/2k))}
{\sin(\beta_2)},~ -\sin^{\alpha} \left( \frac{\beta_2}{2k}\right)\frac{\sin(\alpha\pi/2-\alpha\beta_2/(2k))}{\sin(\beta_2)} \right)
\]
 on $\Gamma_0$ is different from $(a_*, b_*)$. Hence, we have
 $ \beta_2\neq \beta_1.$
However, as we have seen,
 $\gamma \left( \frac{\beta_2+2m\pi}{2k}; \beta_2\right)=(a_*, b_*)=\gamma(\beta_1/2k; \beta_1).$
This means $\gamma(t; \beta_1)$ and $\gamma(t; \beta_2)$ intersects 
in $|a|+b<0$. This contradicts to Lemma \ref{lem:observation}.
Hence, the claim is true.
\end{proof}

\section{No $\tau(0)$-stability}
\label{sec:Notau}

As we have seen, when $k=1$ and $\alpha<1$, the straight line intersects
the continuous boundary curve. This means the numerical solution can be unstable for some parameters $(a, b)\in\mathcal{S}_{*}$. However, the region $\mathcal{R}_1$ is fine, or $|b|+a<0$. 
The trouble happens for some parameters in $\mathcal{R}_{2}$. See the definition for $\mathcal{R}_{1}$ and $\mathcal{R}_{2}$ in \eqref{eq:D1D2}. 
We define the lower part of the numerically stable region as $\mathcal{R}_{2}^h$ with lower boundary curve as $\Gamma_0$, that is, the original lower boundary curve $\Gamma$ is replaced by $\Gamma_0$ in $\mathcal{R}_{2}$; See Figure \ref{Numsta}.
Examples for that $(a, b)\in\mathcal{S}_{*}$ but results unstable numerical solutions includes $(0, b)$ with $-\frac{[(1-\alpha)\pi]^{\alpha}}{2\sin(\pi(1-\alpha)/2)}<b<-2^{\alpha-1}$.
Hence, the analysis for $k=1$ already indicates that the method is not $\tau(0)$-stable. 
As we know, $k=1$ is special since the data for $t<0$ is never used.

A natural question is whether the solution can be absolutely stable for $k$ large enough. Or, will the numerical stability region $\mathcal{S}_k$ contain the stability region for the continuous case? In other words, will the method be $\tau(0)$-stable if we restrict that $k\geq 2$ ?

Unfortunately, the following result, our second main result, gives a negative answer.

\begin{theorem}\label{thm:noabsstability}
Assume that $\alpha\in(0, 1)$, $\tau=1$ and $h=1/k$. For any $k\ge 1$, $\mathcal{R}_1$ always stays in the numerical stability region and there is a portion of $\mathcal{R}_2$ that is outside the numerical stability region $\mathcal{R}_2^{h}$. Consequently, the numerical method is never absolutely stable for $\alpha\in (0, 1)$.
\end{theorem}

The claim regarding $\mathcal{R}_1$ and $k=1$ has already been proved in section \ref{sec:caseI}. We focus on the second claim, which can be proved by the following proposition.

\begin{proposition}\label{pro:Gamma0versusGamma}
For $k\ge 2$, consider the curve $\Gamma_0$. Then, we have the following geometric properties
\begin{enumerate}[(i)]
\item $\Gamma_0$ intersects $a+b=0$ at $A_k:=(a^{(k)}, b^{(k)})$ and we always have $a^{(k)}<\frac{[(1-\alpha)\pi]^{\alpha}\sin(\frac{\alpha\pi}{2})}{\sin(\alpha\pi)}$. Geometrically, this means the vertex of the numerical stability region is above the vertex of $\mathcal{S}_*$.

\item $\Gamma_0$ has an asymptotic line given by
\begin{gather}\label{eq:gamma0asympline}
a-b=2^{\alpha}k^{\alpha}\sin^{\alpha} \left(\frac{\pi}{2k} \right)
\cos\left(\alpha \left(\frac{\pi}{2}-\frac{\pi}{2k} \right) \right).
\end{gather}
 For $\alpha > \alpha_*\approx 0.113$ or $k$ is sufficiently large, this asymptotic line is below the asymptotic line for the curve $\Gamma$, in which cases $\Gamma_0$ intersects with $\Gamma$.
\end{enumerate}
\end{proposition}

\begin{figure}[!ht]
\begin{center}
$\begin{array}{cc}
 \includegraphics[scale=0.65]{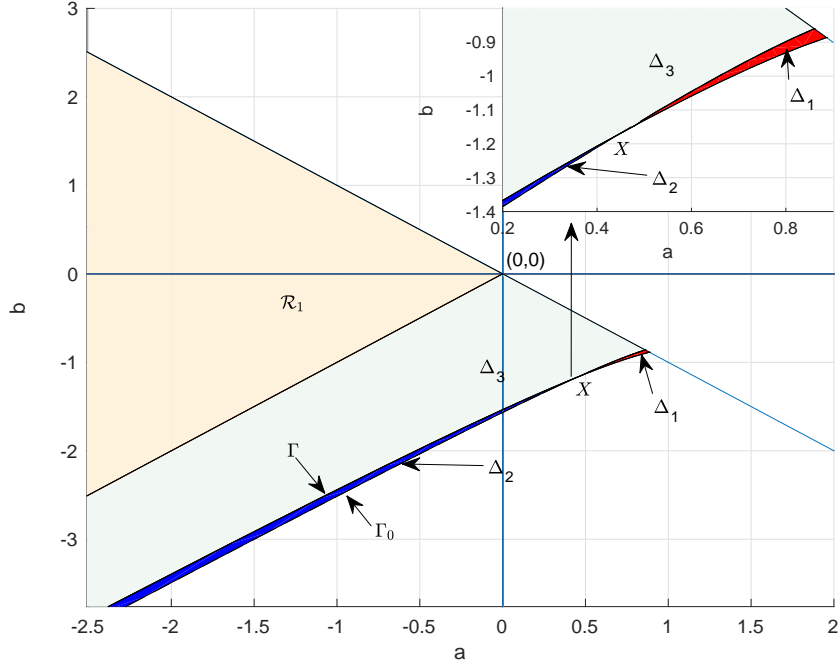}
 \end{array}$
\end{center}
\caption{Comparison of numerical stability region $\mathcal{S}_{h}$ and continuous stability region $\mathcal{S}_*$ for $\tau=1$, $h=0.2, k=5$ and $\alpha=0.5$.}
\label{stacomp}
\end{figure}

We compare the numerical stability region $\mathcal{S}_{h}$ and continuous stability region $\mathcal{S}_*$ with parameter $\tau=1$, $h=0.1, k=10$ and $\alpha=0.5$ in Fig. \ref{stacomp}. We see from this figure that the curves $\Gamma_0$ and $\Gamma$ have an intersection $X=X(a, b)$ with $a>0$ and $b<0$. 
Let's do the following simple division as $\mathcal{R}_2=\Delta_1\cup \Delta_3$ and $\mathcal{R}_2^h=\Delta_2\cup \Delta_3$. $\Delta_1$ is shown in red and $\Delta_2$ is shown in blue. It is found that in $\Delta_1$ the curve $\Gamma_0$ is over $\Gamma$ but in $\Delta_2$ the curve $\Gamma_0$ is blow $\Gamma$. 
See a larger sub-figure in Fig. \ref{stacomp}. The intersection of these two curves $\Gamma_0$ and $\Gamma$ destroys the $\tau(0)$-stability of the numerical method.

\begin{proof}[Proof of Proposition \ref{pro:Gamma0versusGamma} ]

(i).  Consider the intersection between $\Gamma_0$ and $a+b=0$.
We have $2a^{(k)}=2^{\alpha}k^{\alpha}\frac{\cos^{\alpha}(\phi_1)}{\cos(\alpha\phi_1)}$, where
$
\phi_1=\frac{\pi(k-1)}{2(k-\alpha)}.
$
It is clear that
$
\lim_{k\to\infty} a^{(k)}=\frac{[(1-\alpha)\pi]^{\alpha}\sin(\frac{\alpha\pi}{2})}{\sin(\alpha\pi)}.
$
Hence, (i) will follow if we can show that $2a$ is increasing as $k\to\infty$.
For this purpose, we introduce for $\epsilon\in (0, 1/2)$ that
\[
g(\epsilon)=\alpha \left[ \ln \left(\frac{2}{\epsilon}\right)+\ln\sin \left( \frac{(1-\alpha)\pi}{2}
\frac{\epsilon}{1-\epsilon\alpha} \right) \right]
-\ln\cos \left( \frac{\alpha\pi}{2}\frac{1-\epsilon}{1-\epsilon\alpha} \right).
\]
Clearly,  $2a^{(k)}=e^{g(1/k)}$. Hence, we show $g(\e)$ is a decreasing function. Take the derivative
\[
g'(\epsilon)=-\frac{\alpha}{\epsilon}
+\alpha\cot \left( \frac{(1-\alpha)\pi}{2}\frac{\epsilon}{1-\epsilon\alpha} \right) 
\frac{(1-\alpha)\pi}{2}\frac{1}{(1-\epsilon\alpha)^2}
-\cot \left( \frac{\pi}{2}\frac{1-\alpha}{1-\epsilon\alpha} \right) \frac{(1-\alpha)\pi}{2}\frac{\alpha}{(1-\epsilon\alpha)^2}.
\]
Hence, letting
$y=\frac{\pi}{2}\frac{1-\alpha}{1-\epsilon\alpha},$
one has
$g'(\epsilon)=\frac{\alpha y}{(1-\epsilon\alpha)\sin(\epsilon y)}G,$
where
$
G=-\frac{(1-\epsilon \alpha)\sin(\epsilon y)}{\epsilon y}
+\frac{\sin[(1-\epsilon)y]}{\sin y}.
$
We now show that $G<0$. By the inequality $\frac{\sin x}{x}\ge \cos x$, one has
$
G\le \frac{\sin((1-\epsilon)y)}{\sin y}-(1-\epsilon \alpha)\cos\epsilon y
=\frac{\cos y\cos(\epsilon y)}{\sin y}[-\tan(\epsilon y)+\epsilon\alpha \tan y].
$
Now, we set
$z=\alpha \epsilon$ and consider
\[
H(z, \alpha)=-\tan (\epsilon y)+\epsilon\alpha \tan(y)
=-\tan \left( \frac{\pi(z/\alpha-z)}{2(1-z)} \right)+z\tan \left( \frac{\pi(1-\alpha)}{2(1-z)} \right),
\]
where $0<2z\le \alpha\le 1$.
Clearly, $H(z, 1)=0$ and for each $z< \frac{1}{2}$, we can compute
\[
\frac{\partial H}{\partial \alpha}
=\sec^2 \left( \frac{\pi(z/\alpha-z)}{2(1-z)} \right)
\frac{\pi z/\alpha^2}{2(1-z)}+\sec^2 \left( \frac{\pi(1-\alpha)}{2(1-z)} \right)\frac{-z\pi}{2(1-z)}>0.
\]
The last inequality is equivalent to
\[
\cos \left(\frac{\pi(1-\alpha)}{2(1-z)} \right)>\alpha\cos \left(\frac{\pi(z/\alpha-z)}{2(1-z)}\right)
\quad
\text{or}
\quad
\sin \left( \frac{\pi(\alpha-z)}{2(1-z)} \right)>\alpha\sin \left( \frac{\pi(1-z/\alpha)}{2(1-z)} \right).
\]
Since $\beta_*:=\frac{\pi(1-z/\alpha)}{2(1-z)}\in(0, \pi/2)$ and $x\mapsto \sin x$ is concave on $[0, \pi/2]$,
$\sin(\alpha \beta_*)>\alpha\sin(\beta_*)$ and thus the last inequality holds.
Hence, $H<0$ in the region considered and therefore, $G<0$. The proof of (i) is thus complete.

(ii).  For $\Gamma_0$, the value $a-b$ as a function of $\theta$ is decreasing. Moreover, as $\theta\to\pi$, $a\to-\infty$, $b\to-\infty$, $a/b\to 1$ and 
$a(\theta)-b(\theta)\to 2^{\alpha}k^{\alpha}\sin^{\alpha} \left(\frac{\pi}{2k} \right)
\cos\left(\alpha \left(\frac{\pi}{2}-\frac{\pi}{2k} \right) \right)$.
These imply that the asymptotic line of $\Gamma_0$ is given by
\eqref{eq:gamma0asympline}.

Clearly, as $k\to \infty$, this asymptotic line will tend to the asymptotic line of $\Gamma$. 
Consider the function $g(x)=[\sin(x)/x]^{\alpha}\cos(\alpha(\frac{\pi}{2}-x))$, $x\in [0, \pi/4]$.
It can be verified directly that $\frac{d}{dx}\ln g(x)>0$ for $x$ close to $0$, and that $\frac{d^2}{dx^2}\log (g(x))<0$ for all $x\le \frac{\pi}{4}$. This means $\log (g)$ is a concave function.
Consider the value $\alpha^*$ such that $g(\frac{\pi}{4})=g(\frac{\pi}{6})$. This value can be found to be 
 $\alpha^*\approx 0.241$. Then, for $\alpha>\alpha^*$, $2^{\alpha}k^{\alpha}\sin^{\alpha} \left( \frac{\pi}{2k} \right)\cos\left(\alpha \left(\frac{\pi}{2}-\frac{\pi}{2k} \right)\right)$ is decreasing for $k=2,3,\cdots$. This means that the asymptotic line of $\Gamma_0$ is montonely becoming higher and higher as $k$ increases. Hence, they are below $\Gamma$.

There is a further critical value $\alpha_*\in (0, \alpha^*)$ such that when $\alpha> \alpha_*$, the asymptotic lines of $\Gamma_0$ for all $k$ are below the one for $\Gamma$. In fact, due to the concave property of $\ln g$, the possible highest asymptotic line is $k=2$ and $k=\infty$. Hence, $\alpha_*$ can be found easily by solving
\[
4^{\alpha}\sin^{\alpha} \left( \frac{\pi}{4} \right)
\cos\left(\alpha \left( \frac{\pi}{2}-\frac{\pi}{4} \right)\right)=\pi^{\alpha}\cos(\alpha\pi/2).
\]
This gives $\alpha_* \approx  0.113$.

For $\alpha\in (0, \alpha_*)$, when $k$ is small (for example $k=2$), the asymptotic line of $\Gamma_0$ is above that for $\Gamma$. However, since $\frac{d}{dx}\ln g(x)>0$ for $x$ close to $0$, the asymptotic line of $\Gamma_0$ will eventually falls below that for $\Gamma$ when $k$ is large enough. This then verifies the second claim.
\end{proof}

As soon as Proposition \ref{pro:Gamma0versusGamma} is proved, Theorem \ref{thm:noabsstability} is a straightforward corollary, and we omit the proof.

\begin{remark}
According to the proof above, it is also not hard to see that if $\alpha\to 1^-$, $\Gamma_0$ will then fall below $\Gamma$ totally. This implies that the backward Euler scheme is $\tau(0)$-stable for integer DDEs and it is  agrees with the result in \cite{guglielmi1998delay, huang2009delay}.
\end{remark}

\section{Mittag-Leffler numerical stability}
\label{sec:mlstability}

As we have seen from Section \ref{sec:blt} that boundary locus technique can only determine the boundary of the stability region, but cannot determine the asymptotic behavior of the numerical solutions accurately when the numerical solutions is stable. In order to distinguish the integer order DDEs from the F-DDEs, it is very important to describe the long time decay rate of the solutions accurately. To do that, we introduce the technique of the singularity analysis of generating functions.

\subsection{Singularity analysis of generating functions}

The generating functions of $\mu$ and $\omega$ are related by
$\mathscr{F}_{\mu}(z)=\frac{1}{\mathscr{F}_{\omega}(z)}$ if $\omega=\mu^{(-1)}$. 
As usually, we shall often write that $F_{v}(z)\sim f(z)$ as $z\to z_{0}$ that are equivalent in the sense $\lim_{z\to z_{0}}\frac{F_{v}(z)}{f(z)}=1$.

\begin{lemma}\emph{(\cite[Corollary VI.I]{FS09})} 
\label{lem:generating}
 Assume $\mathscr{F}_v(z)$ is analytic on $\Delta(R,\theta) :=\{z: |z|<R, z\neq 1, |\mathrm{arg}(z-1)|>\theta\}$ for some $R>1$ and $\theta\in (0, \frac{\pi}{2})$. If
 $\mathscr{F}_v(z)\sim (1-z)^{-\beta}$ as $z\to 1,  z\in \Delta(R,\theta)$ for $\beta\neq \{0, -1, -2, -3, \cdots\}$, then $v_n \sim \frac{1}{\Gamma(\beta)}n^{\beta-1} $ as $n\to \infty$.
\end{lemma}
Note that this singularity analysis is in fact some Tauberian type results \cite{widder41,FS90}.  The results are further extended to various typical functions in the monograph \cite{FS09}. These results can also be stated equivalently using the discrete Laplace transform. We chose to use the generating function because simply because most of such results are stated with the generating functions, as in \cite{FS90}.
This fundamental lemma allows us to derive the the asymptotic behavior of the function's coefficients  
by studying the function's dominant singularities, which perfectly fits the numerical analysis of the time fractional differential equations. Through the generating functions, one can explicitly derive the algebraic decay rate of numerical solutions, which is a major feature of time fractional order equations. In \cite{wang2020ML}, combined with perturbation analysis, this method was used to establish the numerical Mittag-Leffler stability of fractional LMMs for F-ODEs; i.e., not only the stability region, but also the optimal algebraic decay rate estimate for the numerical solutions. In this section, we make use of this lemma again to establish the Mittag-Leffler numerical stability for F-DDEs.

\subsection{Mittag-Leffler numerical stability}
In order to apply the singularity analysis of generating functions, we first drive the expression for the numerical solutions of F-DDEs by generating functions.
Consider the generating function \eqref{eq:Fy}
\begin{equation}
\mathscr{F}_y(z) =\left( (1-z)^{\alpha}-h^{\alpha} \left( a +bz^{k} \right)  \right)^{-1} \cdot \left( y_{0}(1-z)^{\alpha-1}+h^{\alpha} \left( bg(z)-ay_0 \right)\right),
\end{equation}
where $g(z)=\sum_{\ell=1}^{k-1}y_{\ell-k}z^{\ell}$. We define
\[
P(z)=(1-z)^{\alpha}-h^{\alpha} \left( a +bz^{k} \right).
\]
Clearly, both the numerator and the denominator $P(z)$ are analytical in any $\Delta(R,\theta)$ with $R>1$ and $\theta\in (0, \pi/2)$. To perform the singularity analysis, we need $P(z)$ to have no zeros in some $\Delta(R,\theta)$ region.
Recall the definition of $Q(s)$, we easily see that
\[
P(1-hs)=h^{\alpha}Q(s).
\]

With this relation, we show the following result.
\begin{proposition}
If $(a, b)\in \mathcal{S}_k$ (the open region), then there exists some $R>1$ and $\theta\in (0, \pi/2)$ such that $P(z)$ is nonzero in $\Delta(R,\theta)$. 
\end{proposition}

\begin{proof}
Due to the definition of boundary locus, when $(a, b)\in \mathcal{S}_k$, $Q(s)$ is nonzero on $\mathrm{cl}(D(h^{-1}, h^{-1}))$. Equivalently, $P(z)$ is nonzero on the whole $\mathrm{cl}(D(0, 1))$. Note that $P(z)$ is analytical in any $\Delta(R,\theta)$ and continuous on $\mathrm{cl}(\Delta(R,\theta))$. Fix $\theta$ and set $R_n:=1+1/n\to 1^+$. If there is always a zero point for any $R_n$, then we can find a sequence of zeros $z_n$ that tends to some $z_*\in \partial D(0, 1)$. By the continuity on $\mathrm{cl}(\Delta(R,\theta))$, $P(z_*)=0$. This, however, is impossible. The conclusion follows.
\end{proof}

This simple but surprising finding allows us to take use of the respective advantages of the two methods to establish numerical Mittag-Leffler stability for F-DDEs. 
 According to proposition \ref{pro:asympbehavior}, we know that $\mathscr{F}_y(z)$  is analytical on $\Delta(R,\theta)$, then it follows from Lemma \ref{lem:generating} that 
 \begin{equation}\label{eq:Fydcay}
\begin{split}
\mathscr{F}_y(z)
&=\left( 1-h^{\alpha} \left( a +bz^{k} \right) \mathscr{F}_\mu(z) \right)^{-1} \cdot \left( y_{0}(1-z)^{-1}+h^{\alpha} \left( bg(z)-(ay_0+by_{-k}) \right)\mathscr{F}_{\mu}(z)  \right)\\
&= \frac{ y_{0}(1-z)^{-1}+h^{\alpha} \left( bg(z)-(ay_0+by_{-k}) \right) (1-z)^{-\alpha} } {1-h^{\alpha} \left( a +bz^{k} \right)  (1-z)^{-\alpha} }\\
&= \frac{ y_{0}(1-z)^{\alpha-1}+h^{\alpha} \left( bg(z)-(ay_0+by_{-k}) \right) } {(1-z)^{\alpha}-h^{\alpha} \left( a +bz^{k} \right)   }\\
&\sim - \frac{ y_{0}} {h^{\alpha} \left( a +b\right)   } (1-z)^{\alpha-1}~\text{as} ~z\to 1, \\
\end{split}
\end{equation}
which yields that 
$$
y_{n}\sim - \frac{ y_{0}} {h^{\alpha} \left( a +b\right) } \frac{1}{\Gamma(1-\alpha)} n^{-\alpha}=  - \frac{ y_{0}} { \Gamma(1-\alpha) \left( a +b\right) } t_n^{-\alpha}~\text{ as}~n\to \infty
$$
provided that $a+b\neq 0.$ This shows that the condition $a+b= 0$ is part of the boundary of the numerically stable region. Meanwhile, it is noted that $P(0)\neq 0$ if and only if $a\neq h^{-\alpha}$.

To summarize the above analysis, we get the main results of this section.

\begin{theorem}\label{thm:main}
Let $\alpha\in(0,1)$, $a,b\in\mathbb{R}$,  $k\in\mathbb{N}^{+}$ such that $h=1/k$. 
Then the numerical solutions for the scheme in \eqref{eq:FDDE1} based on the GL method is Mittag-Leffler stable if $(a, b)\in \mathcal{S}_k$, or 
\[
y_{n}\sim - \frac{ y_{0}} { \Gamma(1-\alpha) \left( a +b\right) } t_n^{-\alpha} =O(t_n^{-\alpha}),
\]
 as $n\to\infty$.
\end{theorem}

As mentioned, the results for general $\tau>0$ can be obtained by simple scaling. The corresponding results for $\tau>0$ are given in the introduction (Theorem \ref{thm:MLstab}).

There is some subtlety on the boundary of $\mathcal{S}_k$. In fact, the lower boundary curve $\Gamma_0$ or the straight line corresponds to some zero of $Q(s)$ on $\partial D(h^{-1}, h^{-1})$ that is not $0$. Such parameters is not allowed
for the singularity analysis. However, for $a+b=0$ with $a<\frac{[((1-\alpha)\pi]^{\alpha}\sin(\frac{\alpha\pi}{2})}{\sin(\alpha\pi)}$, the zero of $Q$ corresponds to $s=0$ (or $z=1$ for $P$). There are no other zeros. Then, we can still find some $\Delta(R,\theta)$ such that $P(z)$ is nonzero on it. 
That means the singularity analysis can be applied. With the computation in \eqref{eq:Fydcay}, we find
\[
\mathscr{F}_y(z)\sim y_0(1-z)^{-1}.
\]
From here, one cannot obtain the decay of the solution. Instead, one finds $|y_n|$ to be bounded. In fact, if $\phi(t)=y_0$
for all $t<0$, then $y_n \equiv y_0$. Hence, there is no Mittag-Leffler stability for this case as well. Hence, it seems that 
it is also necessary that $(a, b)\in \mathcal{S}_k$ to have the Mittag-Leffler stability.

The numerical solutions with initial functions $\phi_1(t)=0.4$, $\phi_2(t)=0.1t-0.2$ and $\phi_3(t)=0.3\sin(6t)$ for $\tau=1$, $h=0.05$ and $\alpha=0.8$.

\section{Numerical example}
In this section, we give a simple numerical example to show that the numerical solution is Mittag-Leffler stable when the parameters are inside the numerical stable region, that is, the numerical solution exhibits the optimal polynomial decay rate similar to the continuous model.

In the simulation for the F-DDE model \eqref{eq:F-DDEs}, we take the initial functions $\phi_1(t)=0.4$, $\phi_2(t)=-0.1t-0.2$ and $\phi_3(t)=0.3\sin(6t)$, respectively. The numerical solutions for different stability parameter $(a, b)$ with $\tau=1, h=0.05$ and $\alpha=0.8$ are plotted in Fig. \ref{exam}. For the 
subfigures (a), (b) and (c) where $(a, b)$ lies in the numerical stability region, the numerical solutions keep stable and decay to zero no matter what the initial function is. While in subfigure (d) where $(a, b)$ lies out the numerical stability region, the numerical solutions are not stable and never decay to zero. 

In order to test numerical decay rate quantitatively, we introduce the index function 
\begin{equation} \label{eq:indexp}
\begin{split}
p_{\alpha}(t_{n})=-\frac{\ln(\|y_{n}\|/\|y_{n-1}\|)}{\ln(t_{n}/t_{n-1})}, ~~t_{n}>1.
\end{split}
\end{equation}
The index $p_{\alpha}$ is a numerical observation of $\alpha$ in $\|y_{n}\|=O(t_{n}^{-\alpha})$, which is independent of the initial value functions, see \cite{wang2019dissipativity}.
It shows in Table \ref{tableexam} that the numerical solutions have the polynomial decay rate and the numerical observation $p_{\alpha}$ are fully agree with our theoretical estimate, which is quite different from the exponential decay rate of the solutions to integer DDEs. 

\begin{table}
\caption{Observed $p_\alpha$ with $\tau=1, h=0.1$ and $a=-3, b=1$ for initial function $\phi_1(t)$} 
\begin{center}
\begin{tabular}{lll lll ll }
\hline $t_n$  & $\alpha=0.1$ & $\alpha=0.3$  & $\alpha=0.5$ & $\alpha=0.7$ & $\alpha=0.9$\\
\hline
$100   $ & 0.0896 & 0.2918  & 0.5014 & 0.7069 & 0.9080\\
$200   $ & 0.0902 & 0.2931  &0.5007  & 0.7040  & 0.9041\\
$300  $ & 0.0905 & 0.2938  &0.5005  & 0.7029  & 0.9028\\
$400  $& 0.0907 & 0.2943  & 0.5003 & 0.7023  & 0.9022\\
$500 $& 0.0909 & 0.2946  & 0.5003 & 0.7019  & 0.9017\\
\hline
\end{tabular}
\end{center}
\label{tableexam}
\end{table}

\begin{figure}[htbp]\centering
\subfigure[$a=-3, b=1$]{
\begin{minipage}[t]{0.45\linewidth}
\centering
\includegraphics[width=2.3in]{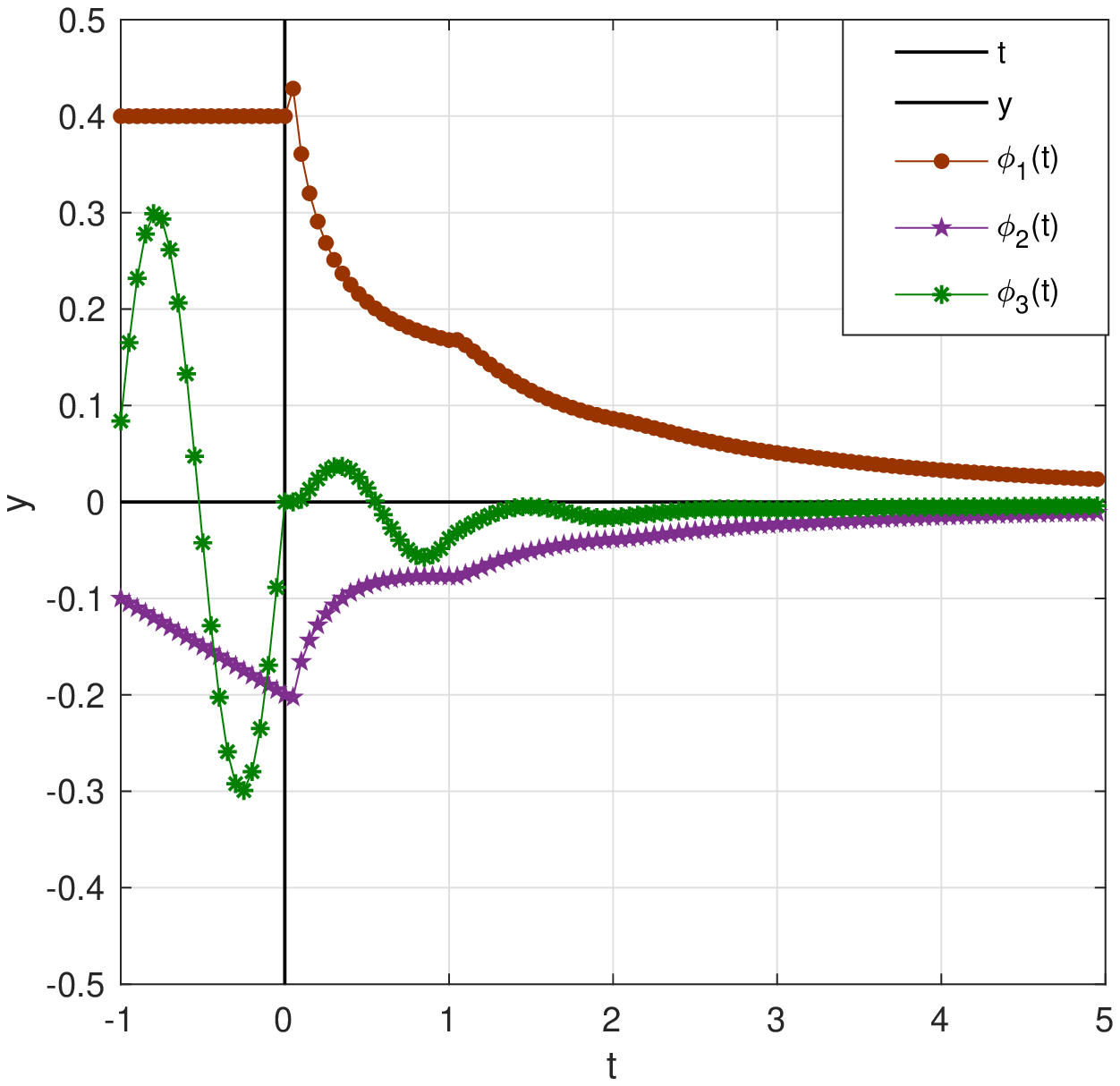}
\end{minipage}
}
\subfigure[$a=-2, b=1.9$]{
\begin{minipage}[t]{0.45\linewidth}
\centering
\includegraphics[width=2.5in]{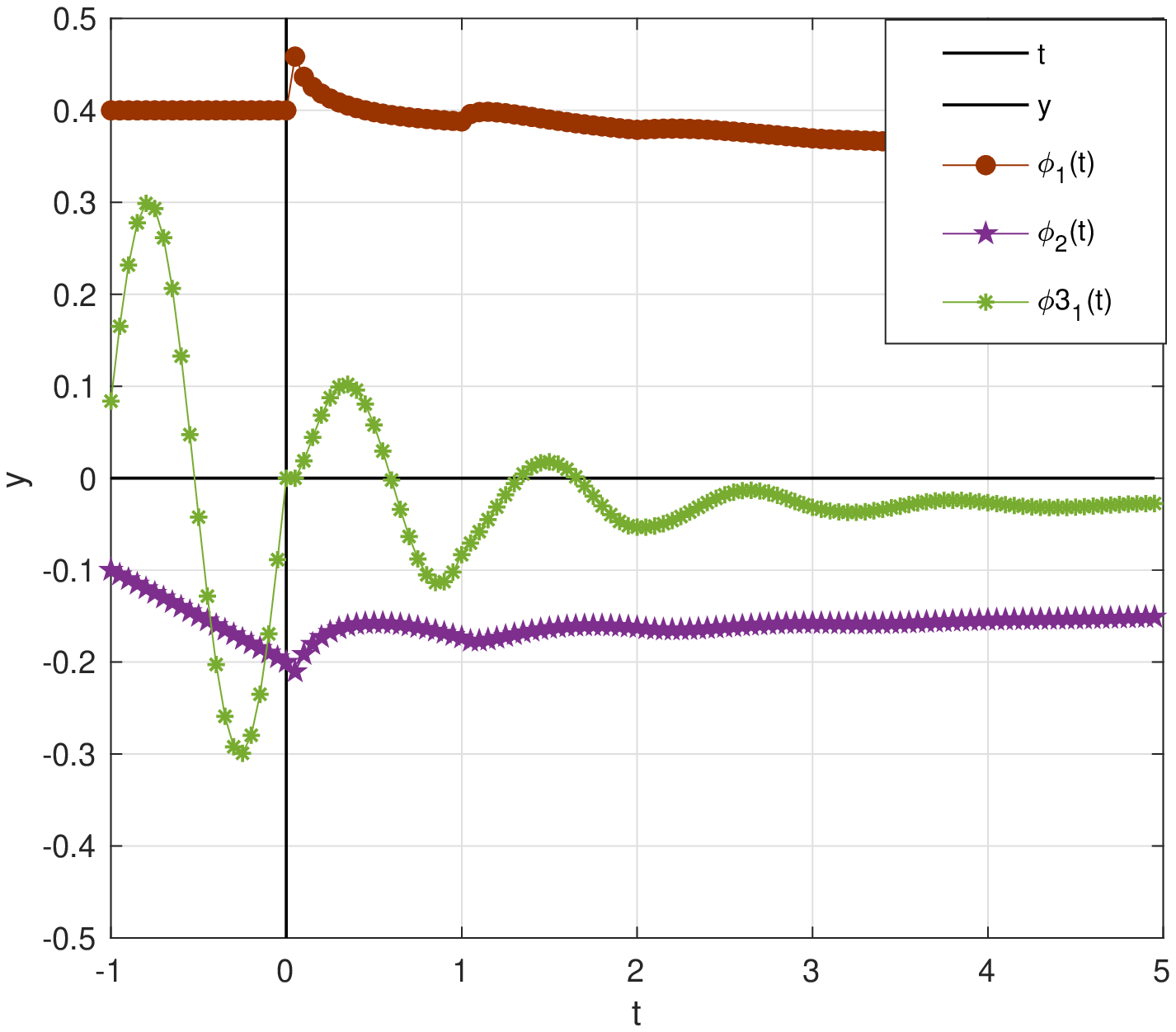}
\end{minipage}
}                 
\subfigure[$a=0, b=-1.5$]{
\begin{minipage}[t]{0.45\linewidth}
\centering
\includegraphics[width=2.4in]{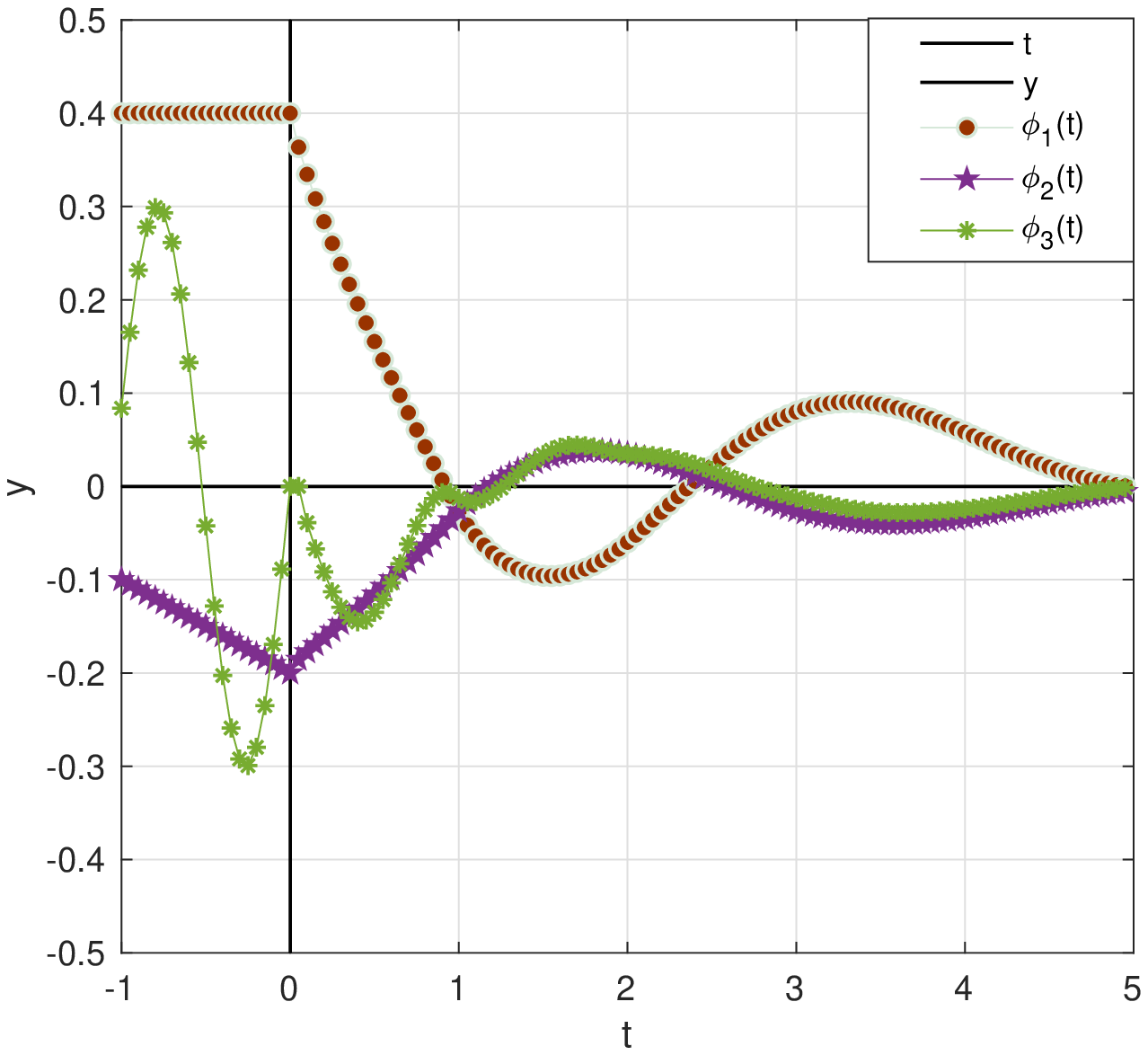}
\end{minipage}
}
\subfigure[$a=0, b=-3$]{
\begin{minipage}[t]{0.45\linewidth}
\centering
\includegraphics[width=2.6in]{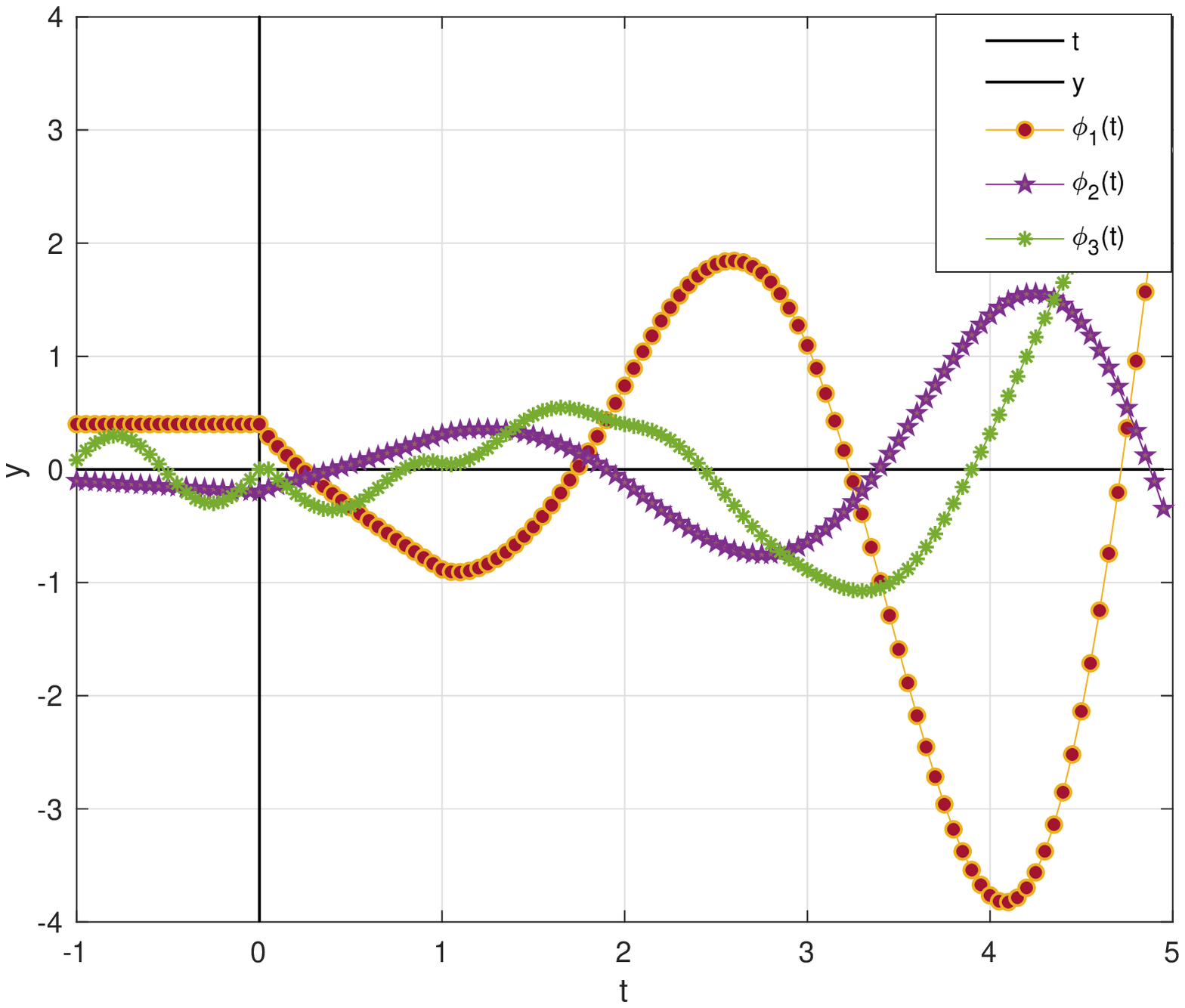}
\end{minipage}
}                    
\caption{ Numerical solutions with different parameter $(a, b)$ for $\tau=1, h=0.05$ and $\alpha=0.8$.}
\label{exam}
\end{figure}

\section*{Acknowledgements} 
The work of L. Li was partially sponsored by NSFC 11901389, 12031013, Shanghai Sailing Program 19YF1421300. 
The work of D. L. Wang was partially sponsored by NSFC 11871057, 11931013.

\appendix

\section{Proof of Lemma \ref{lmm:locuslines}}\label{app:lines}

\begin{proof} [Proof of Lemma \ref{lmm:locuslines}]

For (i), we need to show that
\[
2^{\alpha}<2\frac{[(1-\alpha)\pi]^{\alpha}\sin(\alpha\pi/2)}{\sin(\alpha\pi)}
=\frac{[(1-\alpha)\pi]^{\alpha}}{\sin(\pi(1-\alpha)/2)},
\]
by Lemma \ref{lmm:continuouscurve}.
This is equivalent to
\[
\sin\left(\frac{(1-\alpha)\pi}{2}\right)<\left[\frac{(1-\alpha)\pi}{2}\right]^{\alpha}.
\]
When $\frac{(1-\alpha)\pi}{2}>1$, the inequality clearly holds. If $\frac{(1-\alpha)\pi}{2}\in [0, 1]$, then
\[
\sin\left(\frac{(1-\alpha)\pi}{2}\right)<\frac{(1-\alpha)\pi}{2} \le \left[\frac{(1-\alpha)\pi}{2}\right]^{\alpha}\quad \text{as}\quad \alpha\in (0, 1).
\]

(ii). By Lemma \ref{lmm:continuouscurve}, we only need to show that the straight lines are below the boundary,
\begin{equation}\label{eq:nontaustable}
2^{\alpha}k^{\alpha}\ge \frac{[(1-\alpha)\pi]^{\alpha}}{\sin(\pi(1-\alpha)/2)},~~k\ge 3.
\end{equation}
It suffices to show that
\begin{equation}\label{eq:mid1}
\begin{split}
\sin\left(\frac{\pi(1-\alpha)}{2}\right)- \left[ \frac{(1-\alpha)\pi}{2k} \right]^{\alpha}\ge 0. 
\end{split}
\end{equation}
It is worthing to note that the function $h_1(x)= \frac{\sin x}{x}$ is decreasing with $\frac{2}{\pi}\leq h_1(x)\leq 1$ and it is concave on $[0, \pi/2]$ (note that the sign of the second order derivative is determined by the sign of 
$(1-\frac{x^2}{2})\frac{\sin x}{x}-\cos x$, which is negative for $x\in [0, \pi/2]$).

For $k\ge 5$, since $\sin(\pi(1-\alpha)/2)\ge \frac{\pi}{2}(1-\alpha)\frac{2}{\pi}$,  it suffices to show that
\[
\frac{\pi}{2}(1-\alpha)\frac{2}{\pi}- \left[ \frac{(1-\alpha)\pi}{2k}\right]^{\alpha}\ge 0.
\] 
The latter is equivalent to show that (setting $\beta=1-\alpha$)
$h_2(\beta)=\beta\log\beta-(1-\beta)\log \left( \frac{\pi}{2k} \right)\ge 0.$
The derivative $h'_2(\beta)$ is negative for $k\ge 5$ and the value at $\beta=1$ is zero. Hence, the inequality \eqref{eq:mid1} holds for $k\geq 5$.

If $k=3$, consider directly (setting $\beta=1-\alpha$) that 
$
g(\beta):=\log \left( \sin(\frac{\pi}{2}\beta) \right)+(\beta-1)\log \left( \frac{\beta \pi}{6}\right).
$
As $\beta\to 1^-$,  $g(\beta)$ tends to $0$. Hence, to show $g(\beta)>0$ for $\beta\in(0,1)$, we only need to show that the first order derivative 
\[
g'(\beta)=\frac{\pi}{2}\frac{\cos(\pi \beta/2)}{\sin(\pi\beta/2)}+\log \left(\frac{\pi}{6}\right)+\log(\beta)+(\beta-1)/\beta
\]
is always negative on $\beta\in (0, 1)$. 

It is easy to see that $g'(\beta)<0$ for $\beta\in(0.8, 1)$ (considering only the first two terms).
On the other hand, we have  
$
g''(\beta)=\beta^{-1}+\beta^{-2}-(\pi/2)^2\frac{1}{\sin^2(\pi\beta/2)}.
$
This equation $g''(\beta)=0$ only has one root $\beta_*$ with $\beta\in (0, 1)$ 
and the root satisfies
$
\frac{\sin(\pi \beta_*/2)}{(\pi\beta_*/2)}=\frac{1}{\sqrt{1+\beta_*}}.
$
This can be seen from the fact $h_1(x)=\frac{\sin x}{x}$ is decreasing and concave while the function $\frac{1}{\sqrt{1+x}}$ is convex on $[0, \pi/2]$.
At the same time, we can check $g''(0.8)>0$ and $g''(1)<0$. Hence, we know that the root $\beta_*\in(0.8, 1)$ and $g''(\beta)$ is positive
on $[0, 0.8]$. Together with $g'(0.8)<0$, we find that $g'(\beta)$ is negative on $(0, 0.8]$.  Then the first derivative is always negative on $\beta\in (0, 1)$. 
Therefore, $k=3$ is also proved.
\end{proof}

\section{Proof for the properties of the $\Gamma_0$ curve}\label{app:Gamma0}

\begin{proof}[Proof of Lemma \ref{lem:decrecruv}]
Define $\phi=\theta/2\in \left( \frac{(1-\alpha )\pi}{2(1-\alpha/k)}, \frac{\pi}{2} \right)$.
Consider that
\[
h(\phi):=\frac{\lambda_k(\theta)}{2^{\alpha}k^{\alpha}}=\frac{\sin^{\alpha} \left( \frac{\phi}{k} \right)\sin\left(\phi+\frac{\alpha\pi}{2}
-\frac{\alpha\phi}{k}\right)}{\sin(\phi)}.
\]
To show this function is decreasing, it is sufficient to prove that
\begin{equation}\label{eq:mid2}
 \frac{d}{d\phi}\ln h(\phi)=\frac{\alpha}{k}\frac{\cos(\phi/k)}{\sin(\phi/k)}+\frac{(1-\alpha/k)\cos \left( \phi+\frac{\alpha\pi}{2}-\frac{\alpha\phi}{k}\right)} {\sin \left( \phi+\frac{\alpha\pi}{2}
-\frac{\alpha\phi}{k} \right)}-\frac{\cos\phi}{\sin\phi} \le 0.
\end{equation}
Separate the middle term on the left side of the above equation and use the trigonometric function formula, the above equation can be equivalent to
 $
 \frac{\alpha}{k}\frac{\sin(\phi+\alpha \pi/2-\alpha\phi/k-\phi/k)}{\sin(\phi/k)\sin(\phi+\alpha\pi/2-\alpha\phi/k)}
-\frac{\sin(\alpha \pi/2-\alpha\phi/k)}{\sin(\phi)\sin(\phi+\alpha\pi/2-\alpha\phi/k)}\le 0.
 $
 Noting that $\sin(\phi+\alpha\pi/2-\alpha\phi/k)>0$ under the assumption $\phi\in \left( \frac{(1-\alpha )\pi}{2(1-\alpha/k)}, \frac{\pi}{2} \right)$.
Hence, we need 
\begin{gather}\label{eq:mid3}
\frac{\alpha \sin(\phi+\alpha \pi/2-\alpha\phi/k-\phi/k)}{\sin(\alpha\pi/2
-\alpha\phi/k)}\le \frac{k\sin(\phi/k)}{\sin\phi}.
\end{gather}

We are going to prove that \eqref{eq:mid3} is true for $k\ge 2$. Since
\begin{gather}\label{eq:mid4}
\frac{\alpha \sin(\phi+\alpha \pi/2-\alpha\phi/k-\phi/k)}{\sin(\alpha\pi/2-\alpha\phi/k)}\le 
\frac{\alpha}{\sin(\alpha\pi/2-\alpha \phi/k)}\leq \frac{1}{\cos(\phi/k)},
\end{gather}
where the second inequality is due to the fact $\frac{\alpha}{\sin(\alpha\pi/2-\alpha \phi/k)}$ is increasing function for $\alpha\in (0, 1)$.
By the convexity of $\sin(x)$ on $x\in(0, \pi/2)$, we know that 
$
\sin(\phi)\le \frac{k}{2}\sin(2\phi/k  )
$
 is true for $k\ge 2$, which implies that 
$
\frac{1}{\cos(\phi/k)}\le \frac{k\sin(\phi/k)}{\sin(\phi)}.
$
This together with \eqref{eq:mid4} show that the inequality  \eqref{eq:mid3}
is true for $k\ge 2$. This completes the first part of the proof.

Using the above fact, to show that the line is below $\Gamma_0$, we only have to show that $(a_0,-a_0)$ is above the line, where $a_0$ is defined in Lemma \ref{lem:points}. In other words, we need
\begin{gather}\label{eq:startpoint}
2^{\alpha}k^{\alpha}\ge 2^{\alpha}k^{\alpha}\frac{\cos^{\alpha}(\phi_1)}{\cos(\alpha\phi_1)}, \quad \text{where}\quad \phi_1=\frac{\pi(k-1)}{2(k-\alpha)}.
\end{gather}
Since the right hand side is decreasing in $\phi$, the largest value is achieved at $\phi_1=0$, which is $2^{\alpha}k^{\alpha}$. Hence, the inequality 
\eqref{eq:startpoint} is true.
\end{proof}

\bibliographystyle{alpha}
\bibliography{fracDDE}

\end{document}